\documentclass{mbe}

\usepackage{epstopdf,subfigure}
\usepackage{color}
\usepackage{cite}

\usepackage{txfonts}
\def\typeofarticle{Research article}
\def\currentvolume{}
\def\currentissue{}

\def\ppages{ }
\def\DOI{}
\def\Received{}

\def\Accepted{}
\def\Published{}

\numberwithin{equation}{section}

\newtheorem{theorem}{Theorem}[section]
\newtheorem{lemma}{Lemma}[section]
\newtheorem{remark}{Remark}[section]

\usepackage{graphicx}
\usepackage{epstopdf}
\usepackage{setspace}
\usepackage{mathrsfs}
\usepackage{booktabs}
\usepackage{multirow}
\usepackage{graphicx}
\usepackage{epstopdf}
\usepackage{array}
\usepackage{cite}
\usepackage{indentfirst}
\usepackage{booktabs}
\usepackage{makecell}
\usepackage{hyperref}
\usepackage{tabularx} 
\usepackage{ragged2e} 
\usepackage{booktabs} 
\usepackage{geometry}
\usepackage{float}
\usepackage{chngcntr}
\usepackage{color}

\counterwithout{figure}{section}
\begin{document}

\title{Stability and bifurcation analysis of a two-patch model with
	Allee effect and dispersal}

\author{%
  Yue Xia\affil{1},
  Lijuan Chen\affil{1,}\corrauth,
  Vaibhava Srivastava\affil{2}
  and
  Rana D. Parshad\affil{2,}\corrauth
}

\shortauthors{the Author(s)}

\address{%
  \addr{\affilnum{1}}{School of Mathematics and Statistics, Fuzhou University,
  	Fuzhou, Fujian 350108, China }
  \addr{\affilnum{2}}{Department of Mathematics, Iowa State University,
  	Ames, IA 50011, USA.}}

\corraddr{E-mail: chenlijuan@fzu.edu.cn (Lijuan Chen); rparshad@iastate.edu (Rana D. Parshad)
}

\begin{abstract}
In the current manuscript, a first two-patch model with Allee effect and nonlinear dispersal is presented.
We study both the ODE case and the PDE case here. In the ODE model, the stability of the equilibrium points and the existence of saddle-node bifurcation are discussed. The phase diagram and bifurcation curve of our model are also given by numerical simulation. Besides, the corresponding linear dispersal case is also presented. We show that when the Allee effect is large, high intensity of linear dispersal is not favorable to the persistence of the species. We further show when the Allee effect is large, nonlinear diffusion is more favorable to the survival of the population than linear diffusion. Moreover, the results of the PDE model extends our findings from discrete patches to continuous patches.

\end{abstract}


\maketitle

\section{Introduction}

The conservation of biodiversity is a paramount issue of global scale \cite{Luo2022}. One way to protect endangered species is to create nature reserves or ``refuges" \cite{ref36}, where such species are safe and can breed healthy populations. Studies have shown that remote islands and mountainous regions often provide opportunities to protect endangered species \cite{ref37}. 
However, the destruction of the natural habitat of many species by human activities has resulted in ``fragmentation".
Habitat fragmentation is defined as the breaking up of a large intact area of a single vegetation type into smaller intact units \cite{Alan2002}. This leads to patch-level changes that  can negatively impact species diversity \cite{Keyghobadi2007, Fahrig2002, Fahrig2017, Fahrig2019}.  Also note that dispersal strategy, which is critical for estimating species success under fragmentation, is not well studied in the field \cite{Rohwader2022}. For example, different dispersal speeds may change how easily a species can travel between patches or increase the likelihood of leaving resource-filled patches to avoid competition or predation \cite{Debinski2000}. Numerous studies have shown that building "bridges" between patches that allow groups to communicate with each other can help communities to endure\cite{ref1, ref2}. Thus, studies of dispersal behavior, particularly between habitat patches are informative for the conservation of endangered species \cite{ref5, ref6, ref7}. To this end, the linear diffusion model is well known. For example, in\cite{ref34}, the author proposed the following two models, i.e., with discrete patches,
$$
\frac{\text{d}u^j}{\text{d}t}=u^j\left( a^j-b^ju \right) +\sum_{k=1}^m{D\left( u^k-u^j \right) ,\ j=1,\cdots ,\ m,}
$$
and continuous patches,
$$
\frac{\partial u}{\partial t}=u\left( a-bu \right) +\nabla(D \nabla u),
$$
where $D,\ a^j,\ b^j,\ a$ and $b$ are positive constants. In these models, linear diffusion means that the species can move randomly.

Notice that random movement is reasonable only in some cases (e.g., oceanic plankton\cite{ref35}). Gurney and Nisbet\cite{ref33} studied a biased random motion model as follows.
$$
\frac{\partial u}{\partial t}=ru+D_1\nabla ^2u+D_2\nabla \left( u\nabla u \right) ,
$$
where $r$ is the intrinsic birth rate, $D_1$ is the random dispersal rate and $D_2$  is a positive constant depending upon the proportionality between bias and density gradient. In this model, the authors suggest that the movement of individual populations is largely random, but is influenced to a small extent by the overall distribution of peers. The authors considered that members of a population walk pseudo-randomly in a rectangular network, and the probability distribution of each step is slightly distorted by the local population density gradient, thus the intensity of diffusion is $D_1\nabla ^2u+D_2\nabla \left( u\nabla u \right)$. Later, Allen\cite{ref34} came up with a population modeled by 'pure' biased diffusion, i.e., the discrete patches one
$$
\frac{\text{d}u^j}{\text{d}t}=u^j\left( a^j-b^ju \right) +\sum_{k=1}^m{Du^j\left( u^k-u^j \right) ,\ j=1,\cdots ,\ m,}
$$
and continuous patches,
$$
\frac{\partial u}{\partial t}=u\left( a-bu \right) +\nabla(Du\nabla u).
$$
The biased diffusion model was formulated under the assumption that the population density affects the diffusion rate\cite{ref11, ref12}. In other words, the diffusion rate is governed by the population density. 

The Allee effect \cite{ref15, ref16} plays an important role in population dynamics, and it can be divided into strong Allee effect and weak Allee effect. And strong Allee effect can lead to extinction. Liu et al. \cite{ref22, ref23} studied the influence of strong and weak Allee effect on Leslie-Gower models. Additive Allee effect, an Allee effect involving both strong and weak Allee effect, can be written in the form
$$
\frac{\text{d}u}{\text{d}t}=u\left( 1-u-\frac{a}{m+x} \right).
$$
Lv et al.\cite{ref21} studied the effect of additive Allee effect on an SI epidemic model. In order to protect endangered species, it is particularly important to study the Allee effect on patch models. Chen et al. \cite{ref24} studied the influence of Allee effect of a two-patch model with linear dispersal:
$$
\begin{aligned}
	\frac{\text{d}u}{\text{d}t}&=u\left( 1-u-\frac{a}{m+x} \right) +D_2v-D_1u,\\
	\frac{\text{d}v}{\text{d}t}&=-v+D_1u-D_2v,
\end{aligned}
$$
where $a$ and $m$ are the Allee effect constants. $D_1$ and $D_2$ are dispersal rates of the two patches. Their study suggests that dispersal and Allee effect may lead to the persistence or disappearance of the population in both patches. Wang \cite{ref25} studied the global stability of following two-patch models with Allee effect:
$$
\begin{aligned}
	\frac{\text{d}u}{\text{d}t}=r_1u\left( 1-\frac{u}{K_1}-\frac{a_1}{m_1+u} \right) +D_2v-D_1u,\\
	\frac{\text{d}v}{\text{d}t}=r_2v\left( 1-\frac{v}{K_2}-\frac{a_2}{m_2+v} \right) +D_1u-D_2v,
\end{aligned}
$$
Wang demonstrated that moderate dispersal to the better patch facilitates growth in total population density in the face of strong Allee effects. Many other papers have studied the strong Allee effect and weak Allee effect\cite{ref38, ref39, ref40, ref41}.

In \cite{ref26}, the authors proposed that the following single-species model with Allee effect: $$\frac{\text{d}u}{\text{d}t}=u\left( \frac{ru}{A+u}-d-bu \right),$$ where $r$ is the maximum birth rate, $A$ represents the strength of the Allee effect, $d$ is natural mortality and $b$ denotes the death rate due to intro-prey competition. It is well known that this type of Allee effect is the strong Allee effect, increasing the risk of extinction. It will be interesting to study the effects of diffusion on species with this Allee effect. As we know, there has no study of the Allee effect on a two-patch model with nonlinear dispersal, for this motivation in this paper, we will study a two-patch model with Allee effect and nonlinear dispersal as follows.
\begin{equation}\label{1.1}
	\begin{aligned}
		\frac{\text{d}u}{\text{d}t}&=u\left( \frac{ru}{A+u}-d-bu \right) +Du\left( v-u \right),\\
		\frac{\text{d}v}{\text{d}t}&=v\left( a-cv \right) +Dv\left( u-v \right),
	\end{aligned}
\end{equation}
where $u, v$ are the densities of the population in the first patch and the second patch, respectively. $A$ is the Allee effect constant. $r, d$ and $b$ are birth rate, natural mortality and death rate due to intro-prey competition of population in the first patch, respectively. $a$ and $c$ are intrinsic growth rate and death rate due to intro-prey competition of population in the second patch, respectively. $D$ is the dispersal coefficient. 

It is worth mentioning that human intervention has boosted biodiversity in protected areas from the adversities caused by an Allee effect. Thus, in model (\ref{1.1}), we assume that the population in first patch is affected by the Allee effect, while the population in second patch is free of an Allee effect and is consistent with normal logistic growth. Moreover, it is well known that the natural habitats of many species are fragmented due to human intervention and exploitation. Thus some patches are continuous while others are discrete, so it is important to consider both the ODE and PDE scenarios, while modeling such phenomenon.

As far as we are aware, this is the first time that both nonlinear dispersal and the Allee effect on the population dynamics of a species in a two-patch model has been considered. Although Wang has previously investigated the effect of the strong Allee effect on a patch model in \cite{ref25}, he did not consider the case in which diffusion between patches is nonlinear, and this article would be a good companion study to his research. By comparing the difference between nonlinear and linear dispersal, we conclude that nonlinear diffusion is more conducive to persistence in a fragmented environment. 

The rest of this paper is organized as follows. In Section 2, the ODE case of model (\ref{1.2}) is introduced. And in this section, the existence and stability of equilibrium of model (\ref{1.2}) are proved; the condition for saddle-node bifurcation to occur is proved and the effects of Allee effect and nonlinear dispersal are given. In Section 3, the PDE case of model (\ref{1.2}) is introduced. And the effects of Allee effect and nonlinear dispersal in PDE case are also given. We end this paper with a conclusion in Section 4.

\section{The ODE Case}
\subsection{Existence and stability of equilibrium}
In order to simplify system (\ref{1.1}), let

$$
\bar{u}=\frac{cu}{a},\ \bar{v}=\frac{cv}{a},\ \tau =rt
$$
and
$$
m=\frac{Ac}{a},\ e=\frac{d}{r},\ h=\frac{ab}{cr},\ \delta =\frac{Da}{cr},\ s=\frac{a}{r}.
$$
We still reserve $u$, $v$, $t$ to express $\bar{u},\ \bar{v},\ \tau$, respectively. Then, we get the following simplified system:
\begin{equation}\label{1.2}
	\begin{aligned}
		\frac{\text{d}u}{\text{d}t}&=u\left( \frac{u}{m+u}-e-hu \right) +\delta u\left( v-u \right),\\
		\frac{\text{d}v}{\text{d}t}&=sv\left( 1-v \right) +\delta v\left( u-v \right),
	\end{aligned}
\end{equation}
with the initial conditions: $\ u\left( 0 \right) \geq 0,\ v\left( 0 \right) \geq 0.$ In the above, $0<e<1$ and $m$, $h$, $\delta$, $s$ are all positive constants. The existence and stability of all nonnegative equilibria of model (\ref{1.2}) are proved as follows, respectively.

\textbf{(i)} The trivial equilibrium $E_0\left( 0,\ 0 \right)$ and boundary
equilibrium $E_v\left( 0,\ \frac{s}{s+\delta} \right)$ always exist.

\textbf{(ii)} Existence of the equilibrium $\bar E(\bar u, 0)$ on the $u$ coordinate axis where $\bar u$ satisfies the following equation:
\begin{equation}\label{2.1}
	\begin{aligned}
		\left( h+\delta \right) \bar{u}^2+\left[ m\left( h+\delta \right) +e-1 \right] \bar{u}+me=0.
	\end{aligned}
\end{equation}
If $m\ge \frac{1-e}{h+\delta}$, equation (\ref{2.1}) obviously has no positive root. Following we investigate the case $m<\frac{1-e}{h+\delta}$. Notice that the discriminant of (\ref{2.1}) is $\Delta_1(m)=(h+\delta)^2 m^2-2(1+e)(h+\delta)m+(1-e)^2$. The discriminant of $\Delta_1(m)$ is $\Delta_2=16e(h+\delta)^2>0$. Thus, $\Delta_1(m)=0$ has two positive real roots:
\begin{equation*}
	\begin{aligned}
		m_0:=\frac{\left( 1-\sqrt{e} \right) ^2}{h+\delta},\ m_1:=\frac{\left( 1+\sqrt{e} \right) ^2}{h+\delta}.
	\end{aligned}
\end{equation*}
From $0<e<1$, we can easily get $m_0<\frac{1-e}{h+\delta}<m_1$. Thus, if $m_0<m<\frac{1-e}{h+\delta}$, it follows that $\Delta_1(m)<0$ and then equation (\ref{2.1}) has no positive real root.

\textbf{(iii)} Existence of the positive equilibrium point: from model (\ref{1.2}) we know that the positive equilibrium $E\left( u,\ v \right)$ satisfy the following equation:
\begin{equation*}
	\begin{aligned}
		\left\{ \begin{array}{l}
			\frac{u}{m+u}-e-hu+\delta \left( v-u \right) =0,\\
			s\left( 1-v \right) +\delta \left( u-v \right) =0.\\
		\end{array} \right.
	\end{aligned}
\end{equation*}

Denote $B=\frac{s\delta}{s+\delta}$. The above follows that:
\begin{equation}\label{2.2}
	\begin{aligned}
		\left( h+B \right) u^2+\left[ m\left( h+B \right) +e-1-B \right] u+m\left( e-B \right) =0.
	\end{aligned}
\end{equation}
If $e=B$, equation (\ref{2.2}) becomes
\begin{equation}\label{2.3}
	\begin{aligned}
		u\left[ \left( h+B \right) u+m\left( h+B \right) -1 \right] =0.
	\end{aligned}
\end{equation}
Therefore, if $m\ge \frac{1}{h+B}$, there is no positive equilibrium; if $m<\frac{1}{h+B}$, equation (\ref{2.3}) has a unique positive real root. If $B<e<1$, equation (\ref{2.2}) has a unique positive real root. Following we investigate the case $B<e<1$. Notice that the discriminant of (\ref{2.2}) is $\varDelta _3\left( m \right) =\left( h+B \right) ^2m^2-2\left( e-B+1 \right) \left( h+B \right) m+\left( e-B-1 \right) ^2$. The discriminant of $\varDelta _3(m)$ is $\varDelta _4=16\left( e-B \right) \left( h+B \right) ^2>0$. Thus $\varDelta _3(m)=0$ has two positive real roots:
\begin{equation*}
	\begin{aligned}
		m^*:=\frac{\left( 1-\sqrt{e-B} \right) ^2}{h+B},\ m_{1}^{*}:=\frac{\left( 1+\sqrt{e-B} \right) ^2}{h+B}.
	\end{aligned}
\end{equation*}
From $B<e<1$, we get $m^*<\frac{1+B-e}{h+B}<m_{1}^{*}$. Therefore, if $m^*<m<\frac{1+B-e}{h+B}$, it follows that $\varDelta _3\left( m \right) <0$ and then equation (\ref{2.2}) has no positive real root.

~\\

\begin{theorem} (1) There are two equilibria on the positive coordinate axis of $u$: $E_{\bar{u}_1}\left( \bar{u}_1,\ 0 \right)$ and $E_{\bar{u}_2}\left( \bar{u}_2,\ 0 \right)$ when $0<m<m_0$.

(2) There is a unique equilibrium on the positive coordinate axis of $u$: $E_{\bar{u}_3}\left( \bar{u},\ 0 \right)$ when $m=m_0$. And
\begin{equation*}
	\begin{aligned}
		\bar{u}_1&=\frac{1-e-m\left( h+\delta \right) +\sqrt{\varDelta _1(m)}}{2\left( h+\delta \right)},\ \bar{u}_2=\frac{1-e-m\left( h+\delta \right) -\sqrt{\varDelta _1(m)}}{2\left( h+\delta \right)},\\
		\bar{u}_3&=\frac{1-e-m\left( h+\delta \right)}{2\left( h+\delta \right)}.
	\end{aligned}
\end{equation*}
\end{theorem}

\begin{theorem} (1) If $e<B$, there is a unique positive equilibrium $E_1\left( u_1,\ v_1 \right)$.

\noindent (2) If $e=B$, there is a unique positive equilibrium $E_1\left( u_1,\ v_1 \right)$ when $m<\frac{1}{h+B}$; there is no positive equilibrium when $m\ge \frac{1}{h+B}$.

\noindent (3) If $B<e<1$,

(i) there are two positive equilibria $E_1\left( u_1,\ v_1 \right)$ and $E_2\left( u_2,\ v_2 \right)$ when $m<m^*$;

(ii) there is a unique positive equilibrium $E_3\left( u_3,\ v_3 \right)$ when $m=m^*$;

(iii) there is no positive equilibrium when $m>m^*$. And
\begin{equation*}
	\begin{aligned}
		u_1&=\frac{1+B-e-m\left( h+B \right) +\sqrt{\varDelta _3(m)}}{2\left( h+B \right)}, u_2=\frac{1+B-e-m\left( h+B \right) -\sqrt{\varDelta _3(m)}}{2\left( h+B \right)},\\
		u_3&=\frac{1+B-e-m\left( h+B \right)}{2\left( h+B \right)}.
	\end{aligned}
\end{equation*}
\end{theorem}

Next, we consider the local stability of the equilibrium point. The Jacobian matrix of system (\ref{1.2}) at any point $E\left( u,\ v \right)$ is
\begin{equation}\label{2.4}
	\begin{aligned}
		J_E=\left( \begin{matrix}
			j_{11}&		j_{12}\\
			j_{21}&		j_{22}\\
		\end{matrix} \right),
	\end{aligned}
\end{equation}
where 
$$j_{11}=\frac{\left( 2m+u \right) u}{\left( m+u \right) ^2}-2\left( h+\delta \right) u-e+\delta v,\ j_{12}=\delta u,\ j_{21}=\delta v,\ s-2\left( s+\delta \right) v+\delta u.
$$

\begin{theorem} (1) $E_0\left( 0,\ 0 \right)$ is always a saddle.

\noindent (2) If $B<e<1$, $E_v\left( 0,\ \frac{s}{s+\delta} \right)$ is locally stable; if $e<B$, $E_v\left( 0,\ \frac{s}{s+\delta} \right)$ is a saddle.

\noindent (3) If $e=B$,

(i) $E_v\left( 0,\ \frac{s}{s+\delta} \right)$ is an attracting saddle-node, and the parabolic sector is on the right half-plane when $m>\frac{1}{h+B}$;

(ii)$E_v\left( 0,\ \frac{s}{s+\delta} \right)$ is an attracting saddle-node, and the hyperbolic sector is on the right half-plane when $m<\frac{1}{h+B}$;

(iii) $E_v$ is a stable node when $m=\frac{1}{h+B}$.
\end{theorem}

\begin{proof} (1) From (\ref{2.4}), the Jacobian matrix at $E_0\left( 0,\ 0 \right)$ is
\begin{equation*}
	\begin{aligned}
		J_{E_0}=\left( \begin{matrix}
			-e&		0\\
			0&		s\\
		\end{matrix} \right),
	\end{aligned}
\end{equation*}
it follows that $E_0\left( 0,\ 0 \right)$ is a saddle.

(2) The Jacobian matrix at $E_v\left( 0,\ \frac{s}{s+\delta} \right)$ is
\begin{equation*}
	\begin{aligned}
		J_{E_v\left( 0,\ \frac{s}{s+\delta} \right)}=\left( \begin{matrix}
			B-e&		0\\
			B&		-s\\
		\end{matrix} \right)
	\end{aligned}
\end{equation*}
thus, $E_v\left( 0,\ \frac{s}{s+\delta} \right)$ is a saddle when $e<B$, while $E_v\left( 0,\ \frac{s}{s+\delta} \right)$ is locally stable when $e>B$.

(3) If $e=B$, $J_{E_v\left( 0,\ \frac{s}{s+\delta} \right)}$ has a unique zero eigenvalue. Let  $U_1=u,\ V_1=v-\frac{s}{s+\delta}$, model (\ref{1.2}) can be transformed to the following system:
\begin{equation*}
	\begin{aligned}
		\frac{\text{d}U_1}{\text{d}t}&=U_1\left( \frac{U_1}{m+U_1}-e-hU_1 \right) +\delta U_1\left( V_1+\frac{s}{s+\delta}-U_1 \right),\\
		\frac{\text{d}V_1}{\text{d}t}&=s\left( V_1+\frac{s}{s+\delta} \right) \left( 1-V_1-\frac{s}{s+\delta} \right) +\delta \left( V_1+\frac{s}{s+\delta} \right) \left( U_1-V_1-\frac{s}{s+\delta} \right).
	\end{aligned}
\end{equation*}
Applying the Taylor expansion of $\frac{1}{m+U_1}$ at the origin, it can be rewritten as
\begin{equation}\label{2.5}
	\begin{aligned}
		\frac{\text{d}U_1}{\text{d}t}&=-\left( h+\delta -\frac{1}{m} \right) U_{1}^{2}+\delta U_1V_1+\frac{1}{m^2}U_{1}^{3}+G\left( U_1 \right),\\
		\frac{\text{d}V_1}{\text{d}t}&=BU_1-sV_1-\left( s+\delta \right) V_{1}^{2}+\delta U_1V_1,
	\end{aligned}
\end{equation}
where $G\left( U_1 \right)$ denotes the power series with term $U_{1}^{j}$ satisfying $j>3$. The Jacobian matrix of system (\ref{2.5}) at the origin is
\begin{equation*}
	\begin{aligned}
		J_0=\left( \begin{matrix}
			0&		0\\
			B&		-s\\
		\end{matrix} \right),
	\end{aligned}
\end{equation*}
Then we make the following transformation:
\begin{equation*}
	\begin{aligned}
		\left( \begin{array}{c}
			U_1\\
			V_1\\
		\end{array} \right) =\left( \begin{matrix}
			\frac{s}{B}&		0\\
			1&		1\\
		\end{matrix} \right) \left( \begin{array}{c}
			x_1\\
			y_1\\
		\end{array} \right), t_1=-st,
	\end{aligned}
\end{equation*}
model (\ref{2.5}) becomes

\begin{equation}\label{2.6}
	\begin{aligned}
		\frac{\text{d}x_1}{\text{d}t_1}&=q_0x_{1}^{2}+q_1x_1y_1+q_2x_{1}^{3}+G_1\left( x_1,\ y_1 \right),\\
		\frac{\text{d}y_1}{\text{d}t_1}&=y_1+p_0x_{1}^{2}+p_1x_1y_1+p_2y_{1}^{2}+p_3x_{1}^{3}+G_2\left( x_1,\ y_1 \right),
	\end{aligned}
\end{equation}
where $G_1,\ G_2$ denote the power series with term $x_{1}^{i}y_{1}^{j}$ satisfying $i+j>3$ and
\begin{equation*}
	\begin{aligned}
		q_0&=\frac{B}{m}\left[ m\left( B+h \right) -1 \right] ,\ q_1=-\frac{\delta}{s},\ q_2=\frac{s}{m^2B},
		p_0=-\frac{B}{m}\left[ m\left( B+h \right) -1 \right] ,\\
		p_1&=\frac{2\delta}{s}+1,\ p_2=\frac{\delta}{s}+1,\ p_3=-\frac{s}{m^2B}.
	\end{aligned}
\end{equation*}
If $m\left( h+B \right) >1$ (or $m\left( h+B \right) <1$), we can see that the coefficient of $x_{1}^{2}$  is greater than zero (or less than zero). Applying Theorem 7.1 in \cite{ref27}, we know $E_v\left( 0,\ \frac{s}{s+\delta} \right)$ is an attracting saddle-node, and the parabolic (hyperbolic) sector is on the right half-plane when $q_0>0$ $(q_0<0)$. If $q_0=0$, i.e. $m=\frac{1}{h+B}$, system (\ref{2.6}) becomes
\begin{equation*}
	\begin{aligned}
		\frac{\text{d}x_1}{\text{d}t_1}&=q_1x_1y_1+q_2x_{1}^{3}+G_1\left( x_1,\,\,y_1 \right), \\
		\frac{\text{d}y_1}{\text{d}t_1}&=y_1+p_0x_{1}^{2}+p_1x_1y_1+p_2y_{1}^{2}+p_3x_{1}^{3}+G_2\left( x_1,\,\,y_1 \right).
	\end{aligned}
\end{equation*}
Then we can obtain the implicit function
\begin{equation*}
	\begin{aligned}
		y_1=-p_3x_{1}^{3}+G_3\left( x_1 \right),
	\end{aligned}
\end{equation*}
where $G_3\left( x_1 \right)$ denotes the power series with term $x_{1}^{i}$, $i>3$. Then
\begin{equation*}
	\begin{aligned}
		\frac{\text{d}x_1}{\text{d}t}=q_2x_{1}^{3}+G_4\left( x_1 \right) ,
	\end{aligned}
\end{equation*}
where $G_4\left( x_1 \right)$ denotes the power series with term $x_{1}^{i}$, $i>3$ and $q_2\ne 0$. According to Theorem 7.1 in \cite{ref27} again, and combining the previous time changes, it is clear that $E_v\left( 0,\ \frac{s}{s+\delta} \right)$ is a stable node. \end{proof}

\begin{theorem} (1) Both $E_{\bar{u}_1}\left( \bar{u}_1,\ 0 \right)$ and $E_{\bar{u}_2}\left( \bar{u}_2,\ 0 \right)$ are unstable when $m<m_0$;

\noindent (2) $E_{\bar{u}_3}\left( \bar{u}_3,\ 0 \right)$ is a repelling saddle-node.
\end{theorem}

\begin{proof} From (\ref{2.4}), the Jacobian matrix at $E_{\bar{u}_i}$ is
\begin{equation*}
	\begin{aligned}
		J_{E_{\bar{u}_i}}=\left( \begin{matrix}
			\bar{u}_i\left[ \frac{m}{\left( m+\bar{u}_i \right) ^2}-\left( h+\delta \right) \right]&		\delta \bar{u}_i\\
			0&		s+\delta \bar{u}_i\\
		\end{matrix} \right)
	\end{aligned}
\end{equation*}

(1)  $\theta _{1_i}=\bar{u}_i\left[ \frac{m}{\left( m+\bar{u}_i \right) ^2}-\left( h+\delta \right) \right] \ne 0,\ \theta _{2_i}=s+\delta \bar{u}_i>0$ are two eigenvalues of $J_{E_{\bar{u}_i}}, i=1,2.$ Therefore, both $E_{\bar{u}_1}$ and $E_{\bar{u}_2}$ are unstable.

(2) If $m=m_0$, $J_{E_{\bar{u}_3}}$ has a unique zero eigenvalue. Let $U_2=u-\bar{u}_3,\ V_2=v$, model (\ref{1.2}) can be transformed to the following system:
\begin{equation}\label{2.7}
	\begin{aligned}
		\frac{\text{d}U_2}{\text{d}t}&=\left( U_2+\bar{u}_3 \right) \left( \frac{U_2+\bar{u}_3}{m+U_2+\bar{u}_3}-e-h\bar{u}_3-hU_2 \right) +\delta \left( U_2+\bar{u}_3 \right) \left( V_2-U_2-\bar{u}_3 \right),\\
		\frac{\text{d}V_2}{\text{d}t}&=sV_2\left( 1-V_2 \right) +\delta V_2\left( U_2+\bar{u}_3-V_2 \right),
	\end{aligned}
\end{equation}
Using the same method as Theorem 2.3, system (\ref{2.7}) can be transformed into a form similar to (\ref{2.6}). After complicated calculation, we get $q_0=-\left( h+\delta \right) \frac{\sqrt{e}}{1-\sqrt{e}}<0$. Applying Theorem 7.1 in \cite{ref27}, $E_{\bar{u}_3}\left( \bar{u}_3,\ 0 \right)$ is a repelling saddle-node, Theorem 2.4 is proved. 
\end{proof}

\begin{theorem} (1) $E_1\left( u_1,\ v_1 \right)$ is stable.

\noindent (2) $E_2\left( u_2,\ v_2 \right)$ is always a saddle.

\noindent (3) $E_3\left( u_3,\ v_3 \right)$ is an attracting saddle-node.
\end{theorem}

\begin{proof}

(1)  From (\ref{2.4}), the determinant and the trace of $J_{E_1}$ are
\begin{equation*}
	\begin{aligned}
		Det\left( J_{E_1} \right) &=u_1\left( s+u_1 \right) \left[ h+B-\frac{m}{\left( m+u_1 \right) ^2} \right],\\
		Tr\left( J_{E_1} \right) &=\frac{mu_1}{\left( m+u_1 \right) ^2}-\left( h+2\delta \right) u_1-s.
	\end{aligned}
\end{equation*}
After a simple calculation, we get $Det\left( J_{E_1} \right) >0,\ Tr\left( J_{E_1} \right) <0$, thus, $E_1$ is locally stable.

(2) From Theorem 2.2, if $B<e<1$ and $m<m^*$, then $E_2\left( u_2,\ v_2 \right)$ exists. The determinant of $J_{E_2}$ is
\begin{equation*}
	\begin{aligned}
		Det\left( J_{E_2} \right) =\frac{\left( h+B \right) u_2\left( s+\delta u_2 \right)}{\left( m+u_2 \right) ^2}\left[ \left( m+u_2 \right) ^2-m\left( h+B \right) \right],
	\end{aligned}
\end{equation*}
After a simple calculation, we get $\left( m+u_2 \right) ^2-m\left( h+B \right) <0$, which means $Det\left( J_{E_2} \right) <0$ and $E_2\left( u_2,\ v_2 \right)$ is a saddle.

(3) When $B<e<1$ and $m=m^*$, $E_3\left( u_3,\ v_3 \right)$ exists. Then $\gamma _1=0$ and $\gamma _2=s+\delta u_3>0$ are two eigenvalues of $J_{E_3}$. Using the same method as Theorem 2.3, we know $E_3\left( u_3,\ v_3 \right)$ is an attracting saddle-node, Theorem 2.5 is proved.
\end{proof}

The existence and stability conditions for all equilibria are given in Table 1.

\begin{table}[H]
	\caption{Existence and local stability of all equilibrium.\label{tab1}}
	\newcolumntype{C}{>{\centering\arraybackslash}X}
	\tabcolsep=0.8cm
	\renewcommand\arraystretch{3}
	\resizebox{\columnwidth}{!}{\begin{tabularx}{\textwidth}{ccc}
			\toprule
			\textbf{Equilibrium}	& \textbf{Existence}	& \textbf{Stability}\\
			\midrule
			$E_0\left( 0,\ 0 \right)$		& always			& saddle\\
			$E_v\left( 0,\ \frac{s}{s+\delta} \right)$		& always			&
			\makecell{$B<e<1$,\ stable\\
				$e<B$,\ saddle\\
				$e=B$,\ stable node or\\ 
				attracting saddle-node}
			\\
			$E_{\bar{u}_1}\left( \bar{u}_1,\ 0 \right) ,\ E_{\bar{u}_2}\left( \bar{u}_2,\ 0 \right)$		& $m<m_0$			& unstable\\
			$E_{\bar{u}_3}\left( \bar{u}_3,\ 0 \right)$		& $m=m_0$			& repelling saddle-node\\
			$E_{u_1}\left( u_1,\ v_1 \right)$		& \makecell{$B<e<1, m<m^*$ or \\$e=B, m<\frac{1}{h+B}$ or \\$e<B$ } 			& stable\\
			$E_{u_2}\left( u_2,\ v_2 \right)$		& $B<e<1, m<m^*$			& saddle\\
			$E_{u_3}\left( u_3,\ v_3 \right)$		& $B<e<1, m=m^*$			& attracting saddle-node\\
			
			\bottomrule
	\end{tabularx}}
\end{table}

\begin{theorem} The boundary equilibria $E_v\left( 0,\ \frac{s}{s+\delta} \right)$ is globally asymptotically stable when $B<e<1,\ m>m^*$.
\end{theorem}

\begin{proof}
For model (\ref{1.2}), it is easy to know $\left. \frac{\text{d}u}{\text{d}t} \right|_{u=0}=0,\ \left. \frac{\text{d}v}{\text{d}t} \right|_{v=0}=0,$ which means $u=0$ and $v=0$ are the invariant set of model (\ref{1.2}). Thus, all the solutions of model (\ref{1.2}) are nonnegative. Considering following equations:
\begin{equation*}
	\begin{aligned}
		\frac{\text{d}u}{\text{d}t}\le &u\left( 1-e-hu \right) +\delta u\left( v-u \right),\\
		\frac{\text{d}v}{\text{d}t}=&sv\left( 1-v \right) +\delta v\left( u-v \right).
	\end{aligned}
\end{equation*}
Applying comparison theorem of differential equations, we establish comparison equations:
\begin{equation*}
	\begin{aligned}
		\frac{\text{d}N_1}{\text{d}t}&=N_1\left( 1-e-hN_1 \right) +\delta N_1\left( N_2-N_1 \right),\\
		\frac{\text{d}N_2}{\text{d}t}&=sN_2\left( 1-N_2 \right) +\delta N_2\left( N_1-N_2 \right).
	\end{aligned}
\end{equation*}
From theorem 3.2 in \cite{ref5}, there are positive constants $M_1$ and $T_1$ that make $N_i\left( t \right) \le M$ for $\forall t>T_1$, i=1, 2. So all solutions of model (\ref{1.2}) are uniformly bounded. From Theorem 2.3 and Theorem 2.4, $E_v$ is locally asymptotically stable and there is no positive equilibria when $B<e<1,\ m>m^*$.  Therefore, there exists no limit cycle in the first quadrant. Thus $E_v\left( 0,\ \frac{s}{s+\delta} \right)$ is globally asymptotically stable. The proof of Theorem 2.6 is finished. 
\end{proof}

\begin{theorem} The positive equilibria $E_1\left( u_1,\ v_1 \right)$ is globally asymptotically stable when $e<B$ or $e=B, m<\frac{1}{h+B}$.
\end{theorem}

\begin{proof}
From Theorem 2.3 and Theorem 2.5, the unique positive equilibria $E_1\left( u_1,\ v_1 \right)$ is locally asymptotically stable and the boundary equilibria $E_v\left( 0,\ \frac{s}{s+\delta} \right)$ is unstable in the first quadrant when $e<B$ or $e=B, m<\frac{1}{h+B}$. Considering  the Dulac function $g\left( u,\ v \right) =\frac{1}{u^2v^2}$. Applying $e\le B<s$, we get
\begin{equation*}
	\begin{aligned}
		\frac{\partial \left( gF_1 \right)}{\partial u}+\frac{\partial \left( gF_2 \right)}{\partial v}=\frac{e-s}{u^2v^2}-M<0,
	\end{aligned}
\end{equation*}
where $M=\frac{1}{\left( m+u \right) ^2v^2}+\delta \left( \frac{1}{u^2v}+\frac{1}{uv^2} \right) >0$ and
\begin{equation}\label{2.8}
	\begin{aligned}
		F_1:&=u\left( \frac{u}{m+u}-e-hu \right) +\delta u\left( v-u \right),\\
		F_2:&=sv\left( 1-v \right) +\delta v\left( u-v \right).
	\end{aligned}
\end{equation}
Applying the Bendixson-Dulac discriminant, model (\ref{1.2}) has no limit cycle in the first quadrant. Coordinating the solution of system (\ref{1.2}) is ultimately bounded, we proved $E_1\left( u_1,\ v_1 \right)$ is globally asymptotically stable, the proof of Theorem 2.7 is finished. 
\end{proof}

\begin{remark} Theorem 2.6 shows that the level of the Allee constant $m$ has a large effect on the extinction of the population in the first patch when the intensity of dispersal is low. In detail, when $B<e$, i.e., $\delta<\frac{se}{s-e}(s>e),$ then the species in the first patch will go extinct at any initial value when $m>m^*$. The ecological significance of this result is that when the intensity of dispersal is low, if the birth rate of population in the first patch is affected by the strong Allee effect such that the population face severe difficulties in finding mates, then it will not be able to avoid extinction. However, when the intensity of dispersal is large, i.e., $\delta>\frac{se}{s-e}(s>e)$ which implies $B>e$, the species in both patches will be permanent even though the species in the first patch has strong Allee effect. In other words, nonlinear dispersal can be beneficial to the survival of the species.
\end{remark}

The phase diagram for model (\ref{1.2}) is given in Fig. 1 for the different parameter cases.

\begin{figure}[H]
	\centering
	\subfigure[$e<B$]{
		\includegraphics[width=2in, height=2in]{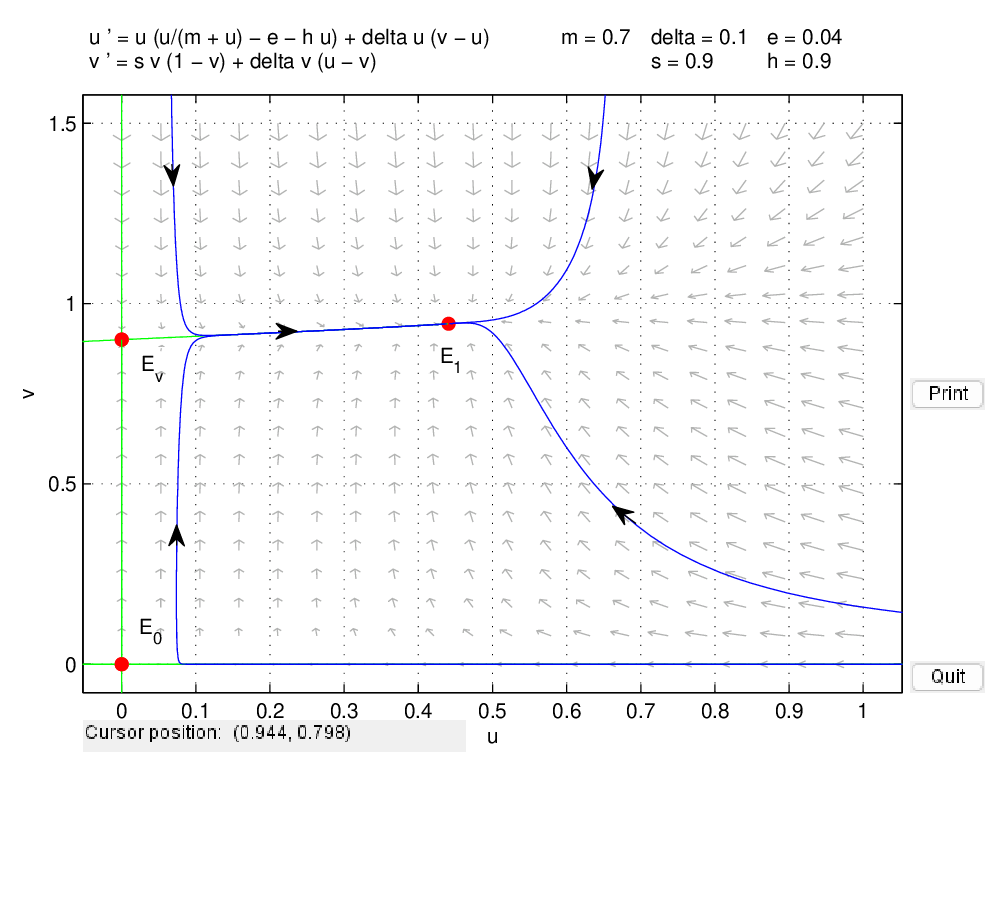}}
	\subfigure[$e=B$, $m>\frac{1}{h+B}$]{
		\includegraphics[width=2in, height=2in]{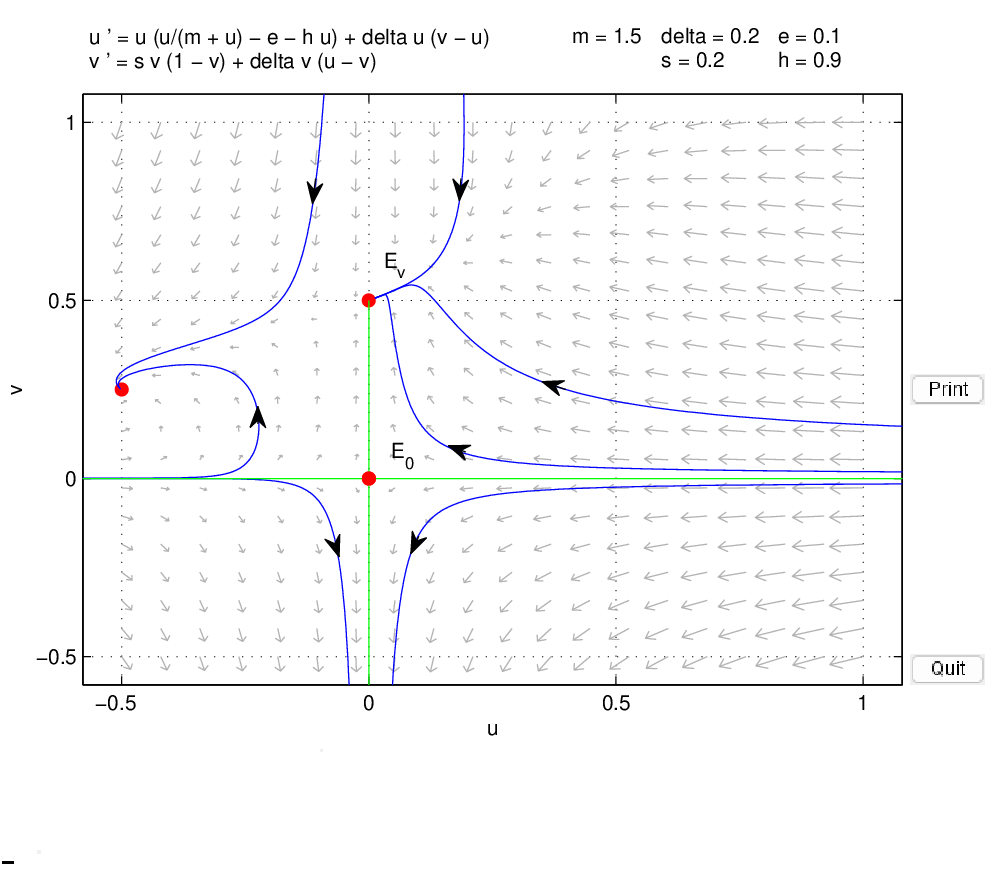}}	
	\subfigure[$e=B$, $m=\frac{1}{h+B}$]{
		\includegraphics[width=2in, height=2in]{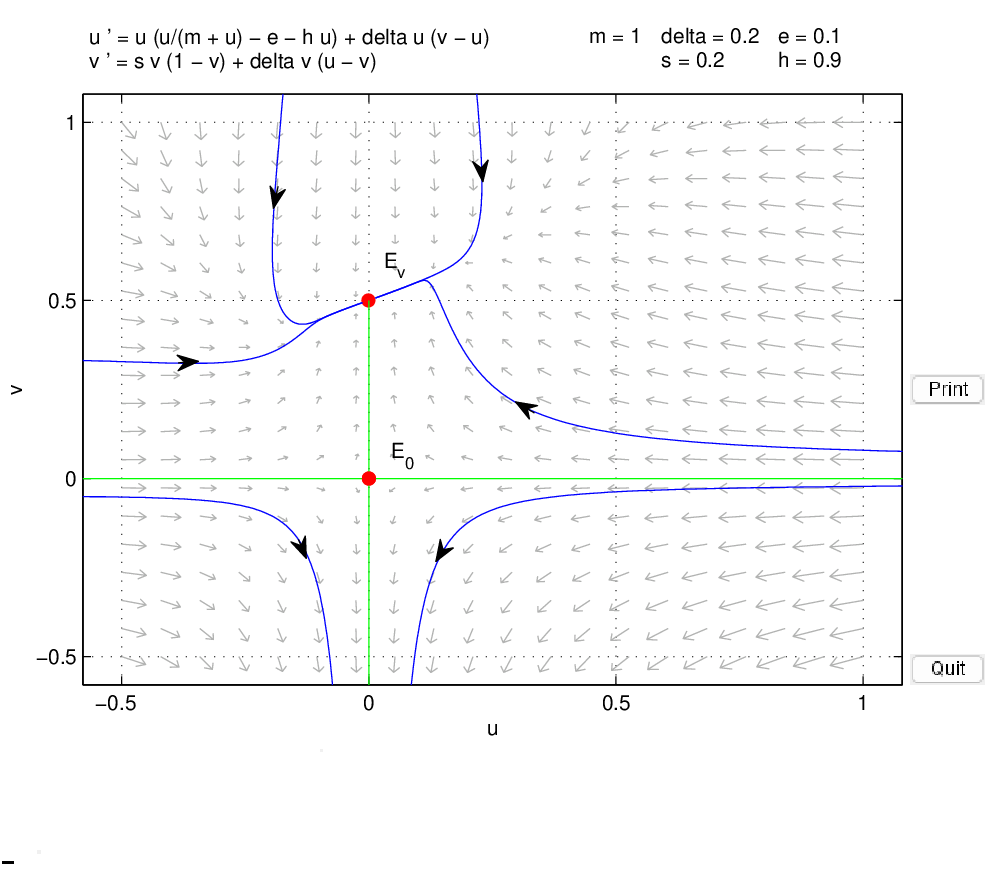}}	
	\subfigure[$e=B$, $m<\frac{1}{h+B}$]{
		\includegraphics[width=2in, height=2in]{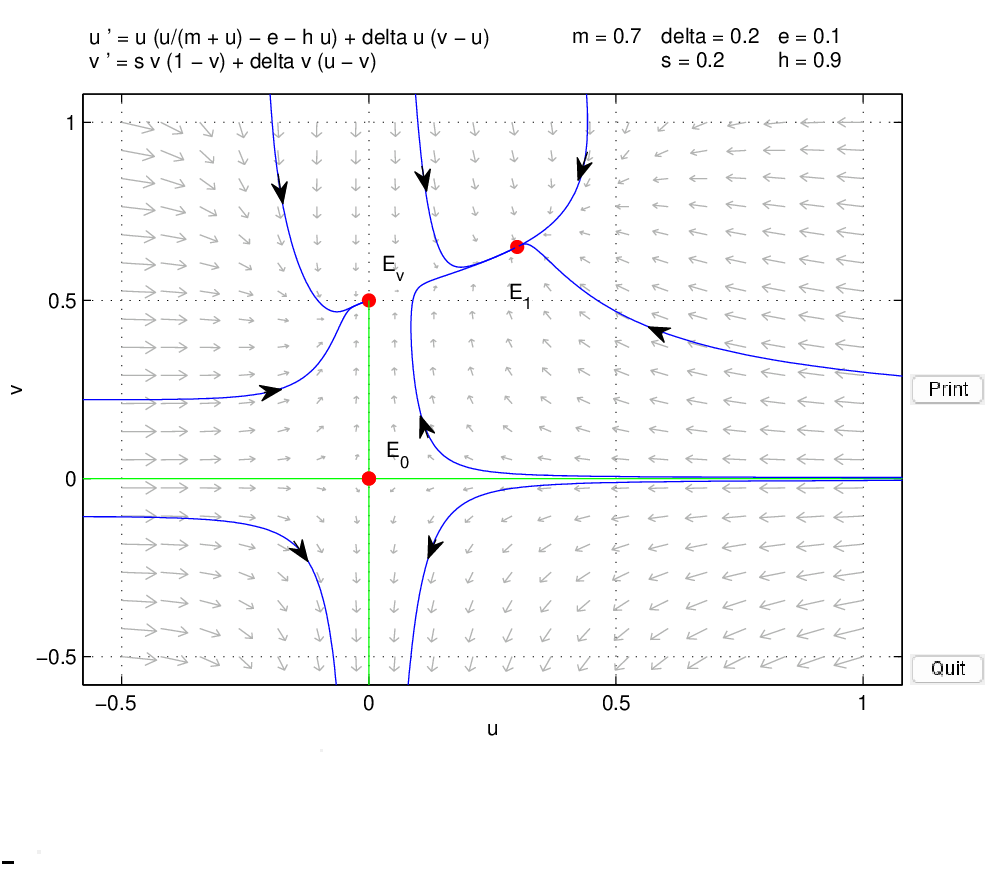}}	
	\subfigure[$B<e<1$, $0<m<m_0$]{
		\includegraphics[width=2in, height=2in]{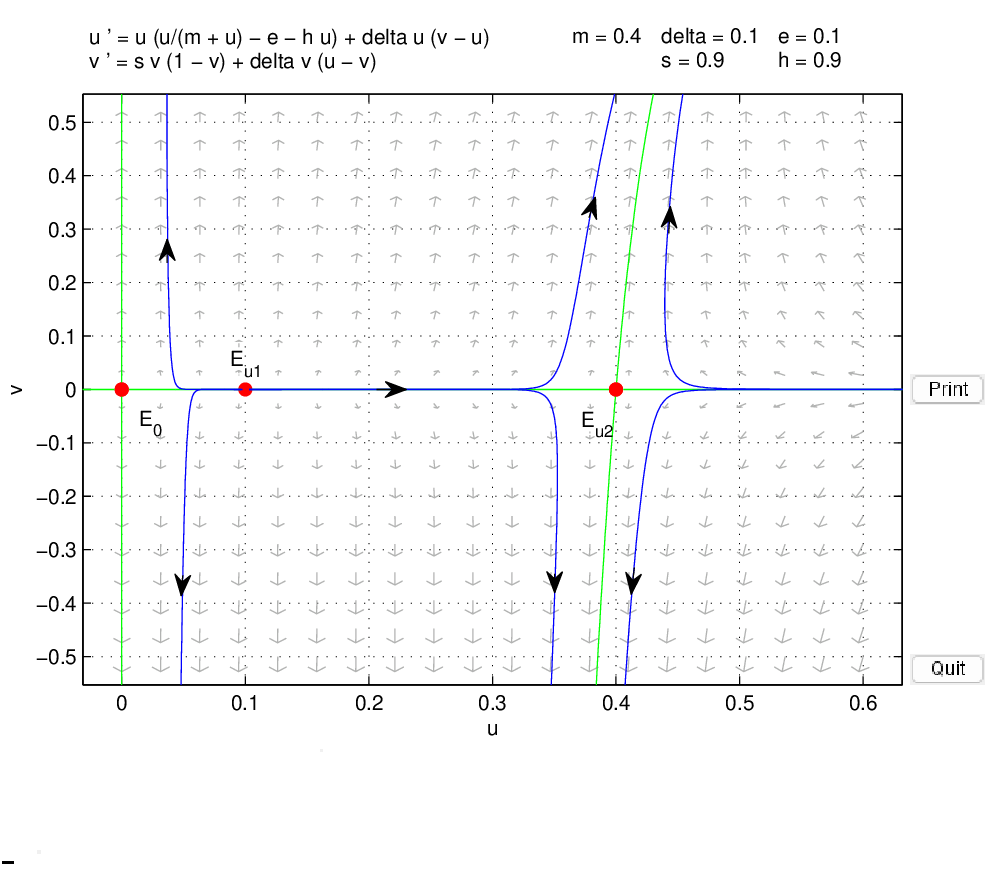}}	
	\subfigure[$B<e<1$, $m=m_0$]{
		\includegraphics[width=2in, height=2in]{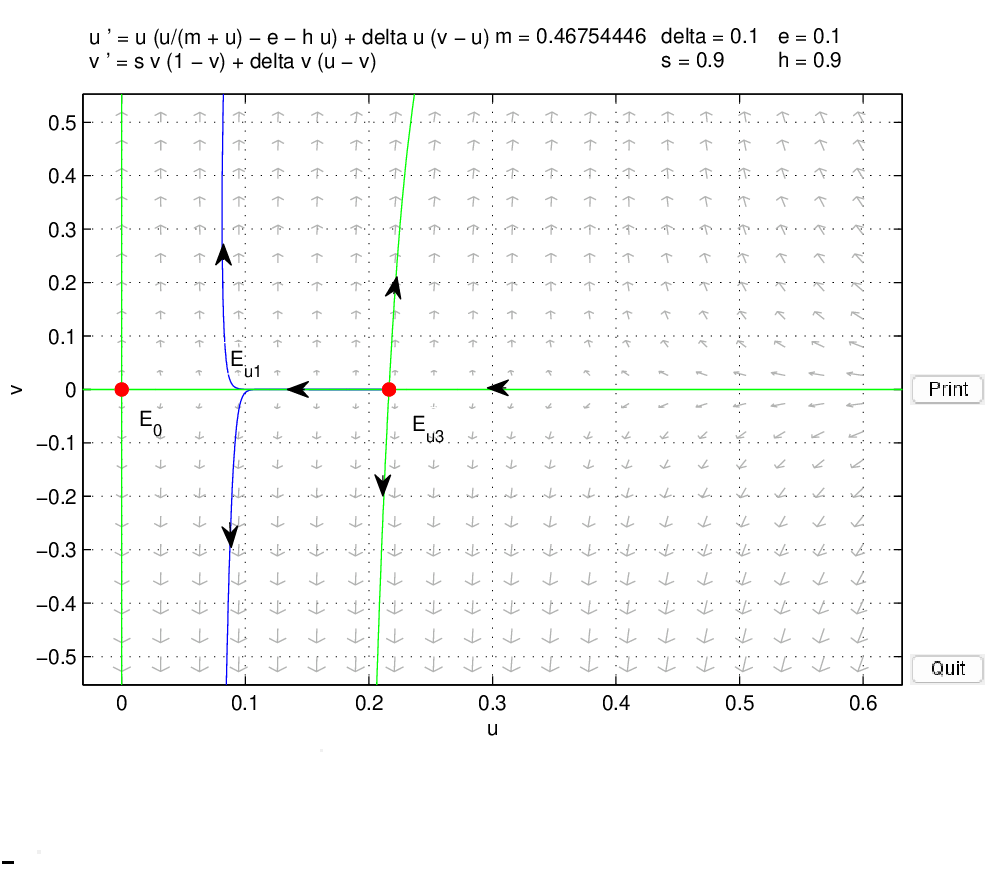}}	
	\subfigure[$e>B$, $m\in \left( m_0,\ m^* \right)$]{
		\includegraphics[width=2in, height=2in]{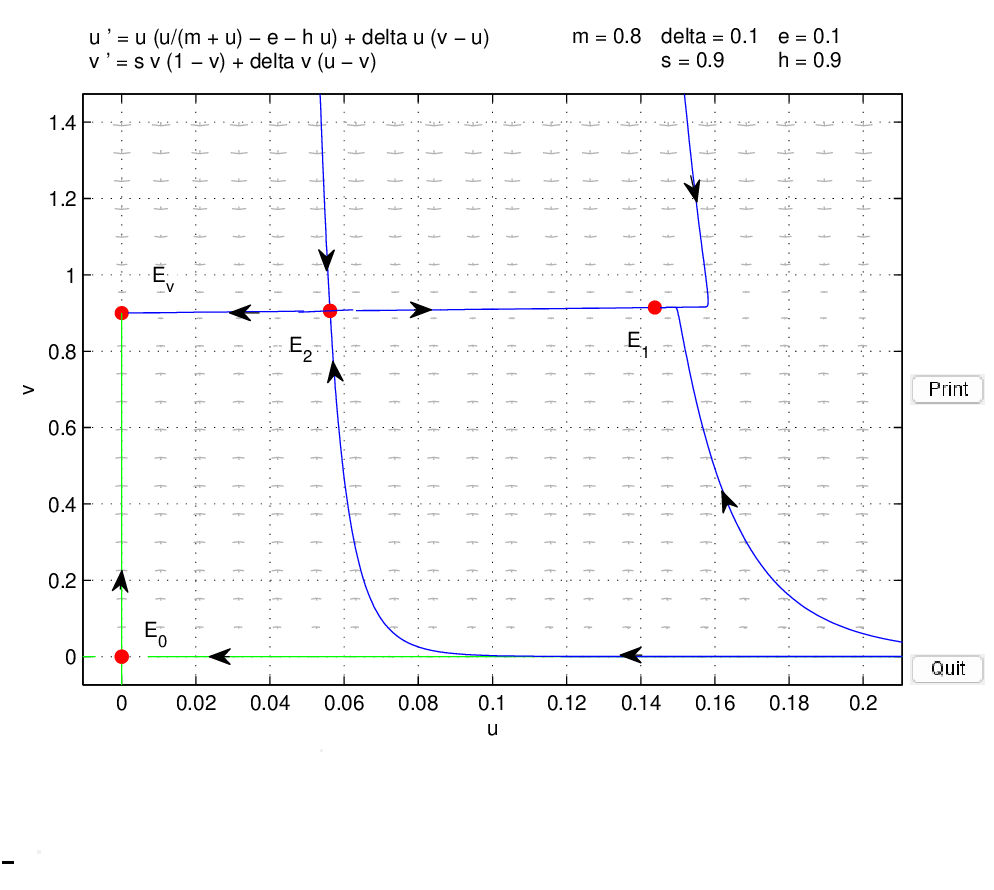}}	
	\subfigure[$e>B$, $m=m^*$]{
		\includegraphics[width=2in, height=2in]{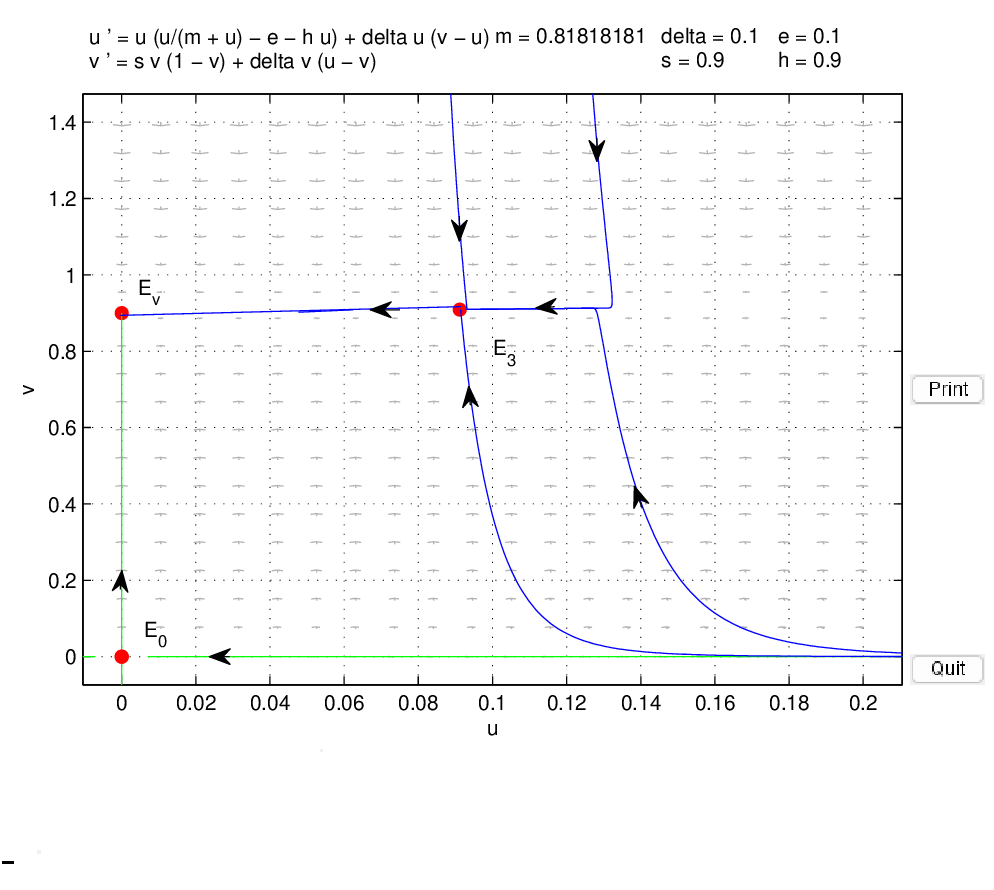}}	
	\subfigure[$e>B$, $m>m^*$]{
		\includegraphics[width=2in, height=2in]{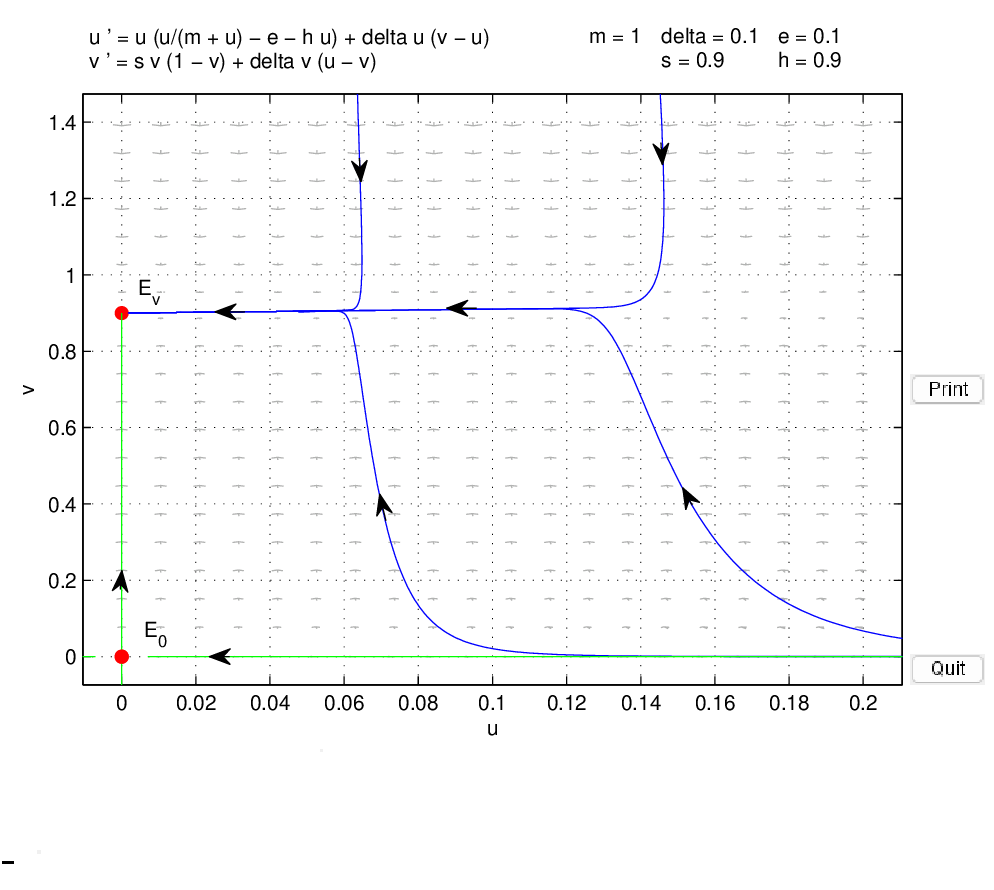}}	
	\caption{\centering {The phase portraits of model (\ref{1.2})}}
\end{figure}

\subsection{Saddle-node bifurcation}
From Theorem 2.2, model (\ref{1.2}) has two positive equilibrium $E_1\left( u_1,\ v_1 \right)$ and $E_2\left( u_2,\ v_2 \right)$ when $B<e<1$, $m<m^*$; However, if $m=m^*$, it has unique positive equilibrium $E_3\left( u_3,\ v_3 \right)$. Saddle-node bifurcation may be raised at here.

\begin{theorem} Saddle-node bifurcation arises at $E_3\left( u_3,\ v_3 \right)$ when $B<e<1$, $m=m^*$.
\end{theorem}

\begin{proof}
From theorem 1.5, we get $Det\left( J_{E_3} \right) =0,\ Tr\left( J_{E_3} \right) =-\left( B+\frac{2\delta ^2}{s+\delta} \right) u_3-s<0$ when $B<e<1$, $m=m^*$. Then $J_{E_3}$ has the unique zero eigenvalue $\gamma _1$. Let
\begin{equation*}
	\begin{aligned}
		\boldsymbol{\alpha :}=\left( \begin{array}{c}
			\alpha _1\\
			\alpha _2\\
		\end{array} \right) =\left( \begin{array}{c}
			1\\
			\frac{\delta}{s+\delta}\\
		\end{array} \right) ,\ \boldsymbol{\beta :}=\left( \begin{array}{c}
			\beta _1\\
			\beta _2\\
		\end{array} \right) =\left( \begin{array}{c}
			1\\
			\frac{\delta u_3}{s+\delta u_3}\\
		\end{array} \right)
	\end{aligned}
\end{equation*}
be the eigenvectors of $J_{E_3}$ and $J_{E_3}^{T}$ corresponding to zero eigenvalue. Next, we have
\begin{equation*}
	\begin{aligned}
		&F_m\left( E_3;\ m^* \right) =\left( \begin{array}{c}
			-\frac{u_{3}^{2}}{\left( m^*+u_3 \right) ^2}\\
			0\\
		\end{array} \right),\\
		&D^2F\left( E_3;\ m^* \right) \left( \boldsymbol{\alpha ,\ \alpha } \right) =\left( \begin{array}{c}
			\frac{\partial ^2F_1}{\partial u^2}\alpha _{1}^{2}+2\frac{\partial ^2F_1}{\partial u\partial v}\alpha _1\alpha _2+\frac{\partial ^2F_1}{\partial v^2}\alpha _{2}^{2}\\
			\frac{\partial ^2F_2}{\partial u^2}\alpha _{1}^{2}+2\frac{\partial ^2F_2}{\partial u\partial v}\alpha _1\alpha _2+\frac{\partial ^2F_2}{\partial v^2}\alpha _{2}^{2}\\
		\end{array} \right) _{\left( E_3;\ m^* \right)}\\
		&\ \ \ \ \ \ \ \ \ \ \ \ \ \ \ \ \ \ \  \ \ \ \ \ \ \ \ \ \ =\left( \begin{array}{c}
			-\sqrt{e-B}\left( h+B \right)\\
			0\\
		\end{array} \right),
	\end{aligned}
\end{equation*}
where $F_1$ and $F_2$ are given in (\ref{2.8}). It is easy to get
\begin{equation*}
	\begin{aligned}
		\boldsymbol{\beta }^TF_m\left( E_3;\,\,m^* \right) &=-\frac{u_{3}^{2}}{\left( m^*+u_3 \right) ^2}\ne 0,\\
		\boldsymbol{\beta }^TD^2F\left( E_3;\,\,m^* \right) \left( \boldsymbol{\alpha ,\,\,\alpha } \right) &=-\sqrt{e-B}\left( h+B \right) \ne 0.
	\end{aligned}
\end{equation*}
Applying Sotomayor theorem \cite{ref28}, system (\ref{1.2}) will arise saddle-node bifurcation at $E_3\left( u_3,\ v_3 \right)$ when $B<e<1$, $m=m^*$, Theorem 2.8 is proved. 
\end{proof}

Besides, saddle-node bifurcation also arises at $E_{\bar{u}_3}\left( \bar{u},\ 0 \right)$ when $m=m_0$, and its proof is analogous to Theorem 2.8, thus we omit it here.

\subsection{Effect of Allee effect and nonlinear dispersal}
From Theorem 2.7, the positive equilibrium point $E_1\left( u_1,\ v_1 \right)$ is globally asymptotically stable when $e<B$ or $e=B, m=\frac{1}{h+B}$. Then, the total population abundance is $T=u_1+v_1$, where $v_1=\frac{s+\delta u_1}{s+\delta}$ and
\begin{equation*}
	\begin{aligned}
		\frac{u_1}{m+u_1}-e-hu_1+B\left( 1-u_1 \right) =0.
	\end{aligned}
\end{equation*}
After a simple derivative calculation, we get
\begin{equation*}
	\begin{aligned}
		\frac{\text{d}u_1}{\text{d}m}&=-\frac{u_1}{C\left( m+u_1 \right) ^2}<0,\\
		\frac{\text{d}v_1}{\text{d}m}&=\frac{\delta}{s+\delta}\frac{\text{d}u_1}{\text{d}m}<0,
	\end{aligned}
\end{equation*}
where $C:=\left( h+B \right) -\frac{m}{\left( m+u_1 \right) ^2}>0$. Thus we have $\frac{\text{d}T}{\text{d}m}=\frac{\text{d}u_1}{\text{d}m}+\frac{\text{d}v_1}{\text{d}m}<0$. The above follows that the stronger the Allee effect, the lower the total population density. Fig. 2 is the bifurcation diagram of parameter $m$, which is done by using MatCont \cite{ref29}.

Considering the subsystem of model (\ref{1.2}) without nonlinear dispersal:
\begin{equation}\label{2.9}
	\begin{aligned}
		\frac{\text{d}u}{\text{d}t}=u\left( \frac{u}{m+u}-e-hu \right).
	\end{aligned}
\end{equation}
For model (\ref{2.9}), it is well known that when the Allee effect is strong, it can lead to population extinction under certain initial values. For example, in Fig. 3, it can be seen that the population of patch 1 becomes extinct in the absence of dispersal. However it becomes persistent with the increase of the dispersal coefficient $\delta$, which reflects the positive effect of dispersal here. In Fig. 4, Allee effect does not lead to extinction when the dispersal coefficient $\delta$ is large enough, which is completely different from the model \label{4.1} in which Allee effect may lead to extinction of the population. Therefore reasonable dispersal is necessary for the conservation of scarce animals.

\begin{figure}[H]
	\centering
    \subfigure[$\left( m,\ u \right)$ plane]{
		\includegraphics[width=2.5in, height=2in]{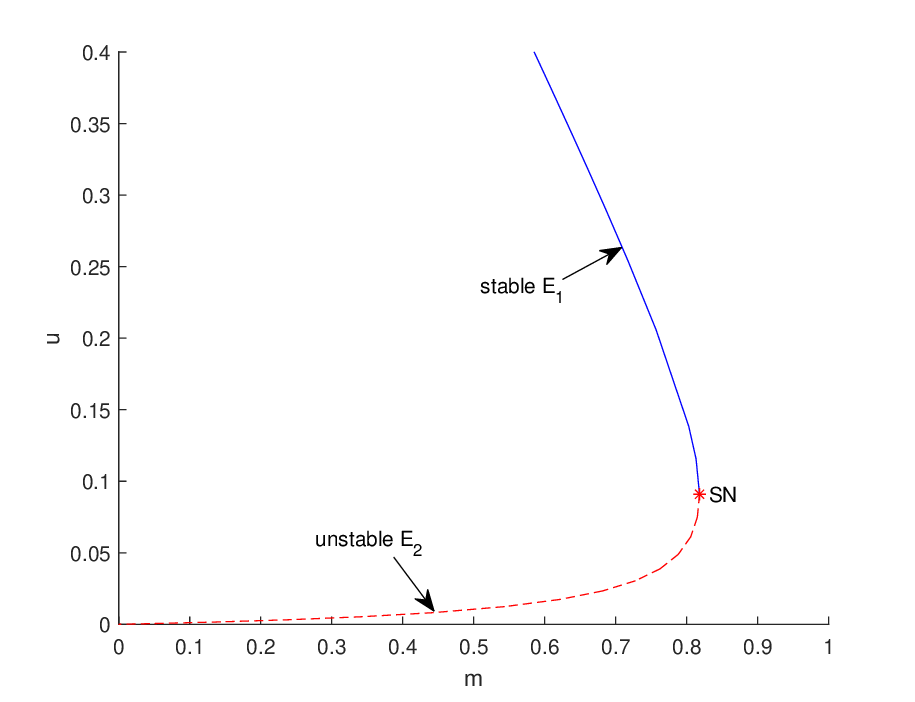}}\hspace{10mm}
	\subfigure[$\left( m,\ v \right)$ plane]{
		\includegraphics[width=2.5in, height=2in]{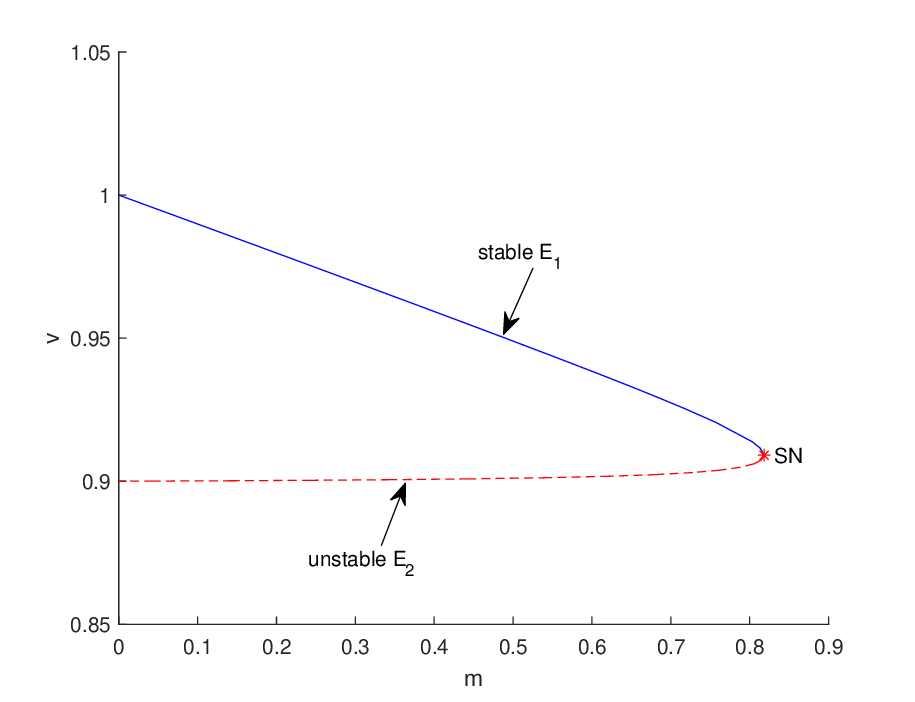}}
	\caption{For model (\ref{1.2}), saddle-node bifurcation diagram of model (\ref{1.2}), where other parameters are $e=\delta=0.1, s=h=0.9$. The blue solid line and the red dashed line indicate the stable and unstable equilibrium points, respectively. SN denotes the saddle-node point.}
\end{figure}

\begin{figure}[H]
	\centering
	\includegraphics[width=0.6\linewidth]{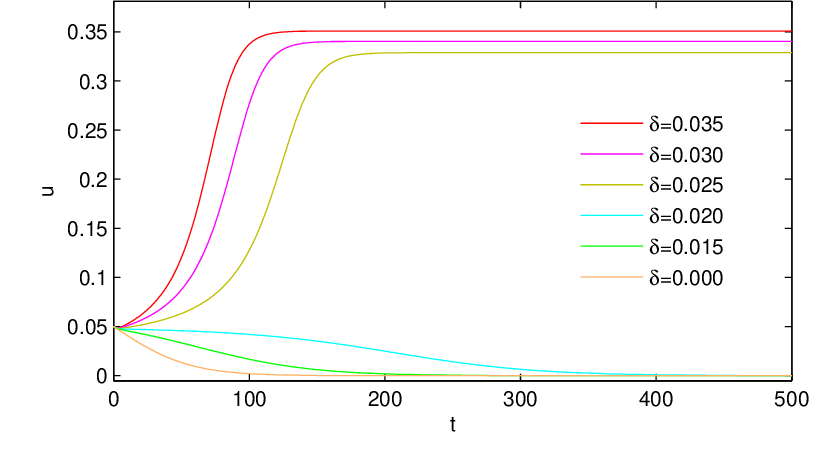}
	\caption{For model (\ref{1.2}), the curves of $u$ over time for different values of $\delta$, where the rest of the parameters are fixed as follows: $m=0.7, e=0.04, h=0.9, s=0.9.$}
\end{figure}

\begin{figure}[H]
	\centering
	\includegraphics[width=0.6\linewidth]{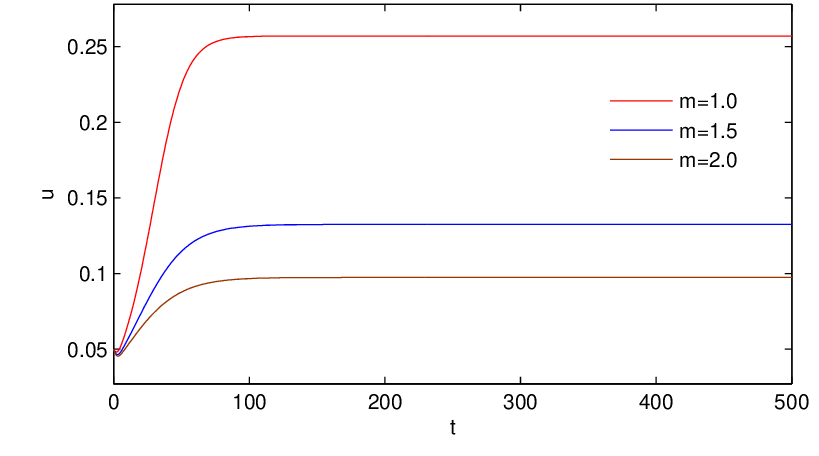}
	\caption{For model (\ref{1.2}), the curves of $u$ over time for different values of $m$ when $B<e<1$. The rest of the parameters are fixed as follows: $\delta=0.1, e=0.04, h=0.9, s=0.9.$}
\end{figure}

In order to better understand the role of nonlinear dispersal, we will present some comparison between nonlinear dispersal and linear dispersal.
Introduce the model with linear dispersal as follows. 
\begin{equation}\label{2.10}
	\begin{aligned}
		\frac{\text{d}u}{\text{d}t}&=u\left( \frac{u}{m+u}-e-hu \right) +\delta \left( v-u \right),\\
		\frac{\text{d}v}{\text{d}t}&=sv\left( 1-v \right) +\delta \left( u-v \right)
	\end{aligned}
\end{equation}
For model (\ref{2.10}), up to now, we can not present the complete qualitative analysis such as the sufficient and necessary condition of the existence of positive equilibrium. We will obtain two sufficient conditions which ensure that system (\ref{2.10})  does not have positive equilibrium and there is an unique positive equilibrium, respectively. And the complete qualitative analysis will be our future work.

~\\
\begin{theorem} For model (\ref{2.10}), if $m\ge \frac{1}{h},\ e>s$ and $\delta>\frac{es}{e-s}$, then there has no positive equilibrium and the trivial equilibrium $O\left( 0,\ 0 \right)$ is globally asymptotically stable.
\end{theorem}

\begin{proof}
Let the right-hand sides of model (\ref{2.10}) equal to zero to get
\begin{equation}\label{2.11}
	\begin{aligned}
		u&=\frac{1}{\delta}\left[ \left(-s+\delta \right) v+sv^2 \right] :=H_1\left( v \right),\\
		v&=\frac{1}{\delta}u\left(-\frac{u}{m+u}+e+\delta +hu \right) :=H_2\left( u \right).
	\end{aligned}
\end{equation}
It is easy to follows that $H_1\left( v \right)$ is strictly monotonously increasing and concave. And $H_1\left( 0 \right) =0, H_{1}^{'}\left( 0 \right) =\frac{\delta -s}{\delta}>0, \ H_2\left( 0 \right) =0$. Besides, we can obtain
\begin{equation*}
	\begin{aligned}
		H_{2}^{'}\left( u \right) =-\frac{1}{\delta}\left[ \frac{u\left( 2m+u \right)}{\left( m+u \right) ^2}-e-\delta -2hu \right] ,\ H_{2}^{'}\left( 0 \right) =\frac{e+\delta}{\delta}>0.
	\end{aligned}
\end{equation*}
And if $h\geq\frac{1}{m}$, then
\begin{equation*}
	\begin{aligned}
		H_{2}^{''}\left( u \right) =\frac{2}{\delta}\left[ h-\frac{m^2}{\left( m+u \right) ^3} \right] >H_{2}^{''}\left( 0 \right) =\frac{2}{\delta}\left( h-\frac{1}{m} \right) \ge 0.
	\end{aligned}
\end{equation*}
Thus $H_2\left( u \right)$ is also strictly monotonously increasing and concave. The above follows that if $H_{1}^{'}\left( 0 \right) H_{2}^{'}\left( 0 \right) >1$, two curves $u=H_1(v)$ and $v=H_2(u)$ will not intersect each other in the first quadrant which implies that model (\ref{2.10}) has no positive equilibrium. And the diagrams of curves $H_1\left( v \right)$ and $H_2\left( u \right)$ are shown in Figure 6. Notice that when  $e>s$ and $\delta \ge \frac{es}{e-s}$, it follows that $H_{1}^{'}\left( 0 \right) H_{2}^{'}\left( 0 \right) >1$.
Summarizing the above, we can conclude that $m\ge \frac{1}{h},\ e>s, \ \delta>\frac{es}{e-s}$, then model (\ref{2.10}) has no positive equilibrium and then it only has the trivial equilibrium $O\left( 0,\ 0 \right)$.

Next, we will consider the local  asymptotically stability of the trivial equilibrium $O\left( 0,\ 0 \right)$ of model (\ref{2.10}). The Jacobian matrix at $O\left( 0,\ 0 \right)$ is
\begin{equation}\label{2.12}
	\begin{aligned}
		J_O=\left( \begin{matrix}
			-e-\delta&		\delta\\
			\delta&		s-\delta\\
		\end{matrix} \right).
	\end{aligned}
\end{equation}
And
\begin{equation*}
	\begin{aligned}
		Det\left( J_O \right) &=\left( e-s \right) \delta -se>0,\\
		Tr\left( J_O \right) &=-\left( e+s+2\delta \right) <0.
	\end{aligned}
\end{equation*}
Therefore $O\left( 0,\ 0 \right)$ is locally asymptotically stable. Also model (\ref{2.10}) has no limit cycle since it has no positive equilibrium when $m\ge \frac{1}{h},\ e>s$ and $\delta>\frac{es}{e-s}$. The above follows that $O\left( 0,\ 0 \right)$ is globally asymptotically stable. Theorem 2.9 is proved.  
\end{proof}

\begin{figure}[H]
	\centering
    \subfigure[$\left( t,\ u \right)$ plane]{
		\includegraphics[width=2.8in, height=2in]{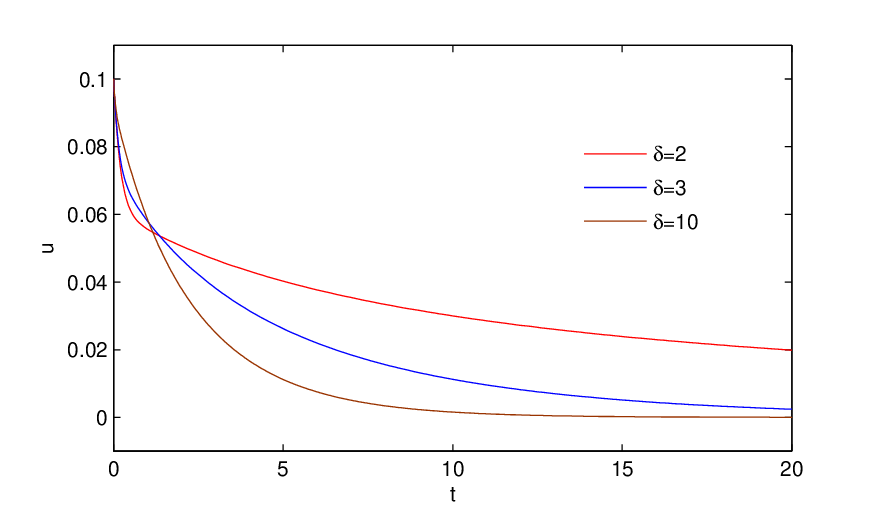}}
	\subfigure[$\left( t,\ v \right)$ plane]{
		\includegraphics[width=2.8in, height=2in]{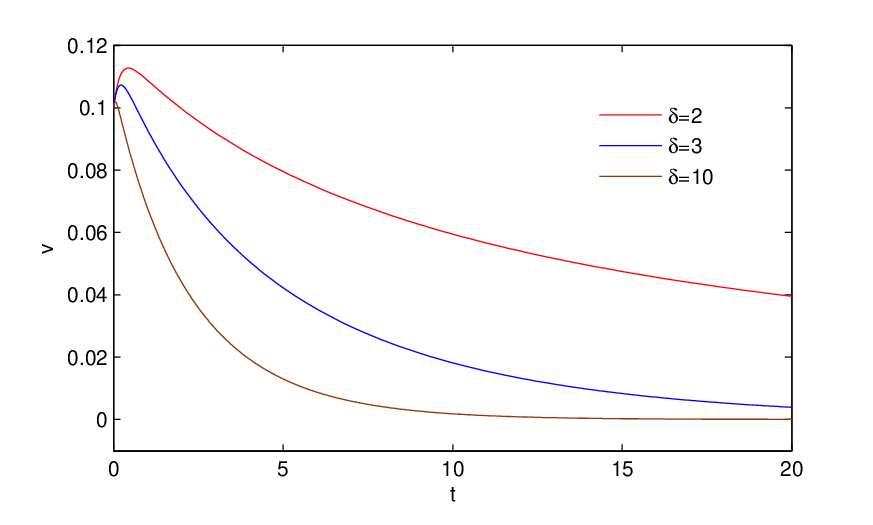}}
	
	\caption{For model (\ref{2.10}), the curves of $u$ and $v$ over time of model (\ref{2.10}) for different values of $\delta$ where the rest of the parameters are fixed as follows: $m=2, e=2, h=1, s=1$.}
\end{figure}

\begin{figure}[H]
	\centering
	\subfigure[]{
		\includegraphics[width=2.5in, height=2in]{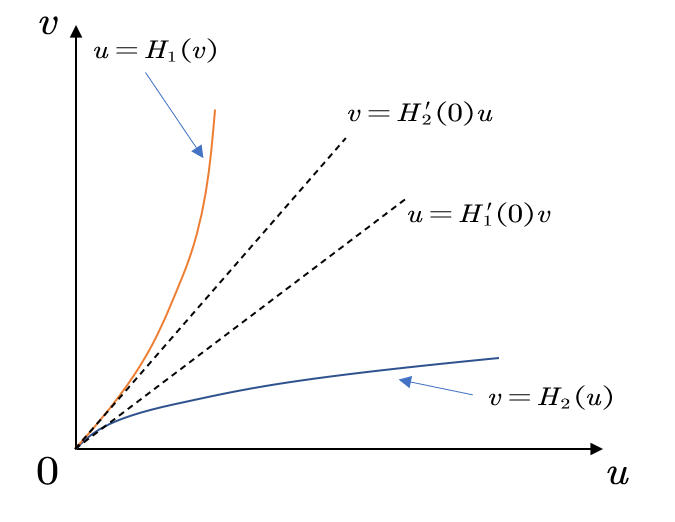}}
    \subfigure[]{
	\includegraphics[width=2.5in, height=2in]{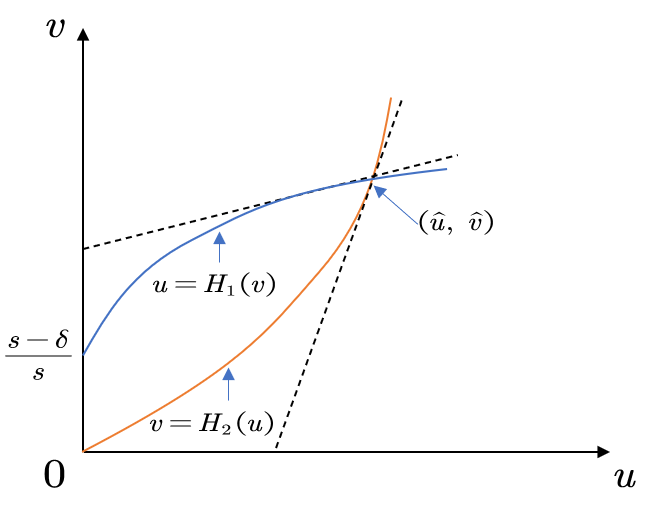}}
		\caption{Number of intersections between $u=H_1\left( v \right)$ and $v=H_2\left( u \right)$ in the first quadrant. (a) no intersection; (b) one intersection.}
\end{figure}

\begin{theorem}
For model (\ref{2.10}), if $m\ge \frac{1}{h}$, $0<\delta \le s$, then it exists a unique positive equilibrium $\hat{E}\left( \hat{u},\ \hat{v} \right)$, where $\hat{u}$ and $\hat{v}$ satisfy the equations (\ref{2.11}). And $\hat{E}\left( \hat{u},\ \hat{v} \right)$ is globally asymptotically stable when $0<\delta <\frac{s-e}{2}$.
\end{theorem}

\begin{proof}
From the proof of Theorem 2.9, if $m\geq \frac{1}{h}$, we can obtain that both $H_1\left( v \right) $ and $H_2\left( u \right)$ are strictly monotone increasing and concave for $u>0$ and $v>0$. From $H_1\left( \frac{s-\delta}{s} \right) =H_2\left( 0 \right) =0$ and $H_{1}^{'}\left( \frac{s-\delta}{s} \right) =\frac{s-\delta}{\delta}\ge 0,\ H_{2}^{'}\left( 0 \right) >0$, then two curves $u=H_1(v)$ and $v=H_2(u)$ will intersect each other once in the first quadrant which implies that model (\ref{1.2}) has a unique positive equilibrium. And it is shown in Figure 6 (b). And a simple calculation gives us $H_{1}^{'}\left( \hat{v} \right) H_{2}^{'}\left( \hat{u} \right) >1$.

From (\ref{2.12}), then we get $Det\left( J_O \right) =e\left( \delta -s \right) -\delta s<0$. Thus $O\left( 0,\ 0 \right)$ is a saddle. Next, we will consider the stability of the positive equilibrium $\hat{E}\left( \hat{u},\ \hat{v} \right)$. From (\ref{2.11}), model (\ref{2.10}) can rewritten as
\begin{equation*}
	\begin{aligned}
		\frac{\text{d}u}{\text{d}t}=\delta \left[ v-H_2\left( u \right) \right]:=Q_1 ,\\
		\frac{\text{d}v}{\text{d}t}=\delta \left[ u-H_1\left( v \right) \right]:=Q_2 .
	\end{aligned}
\end{equation*}
The Jacobian matrix at $\hat{E}\left( \hat{u},\ \hat{v} \right)$ is
\begin{equation*}
	\begin{aligned}
		J_{\hat{E}}=\left( \begin{matrix}
			-\delta H_{2}^{'}\left( u \right)&		\delta\\
			\delta&		-\delta H_{1}^{'}\left( v \right)\\
		\end{matrix} \right).
	\end{aligned}
\end{equation*}
And
\begin{equation*}
	\begin{aligned}
		Det\left( J_{\hat{E}} \right) =\delta ^2\left[ H_{1}^{'}\left( \hat{v} \right) H_{2}^{'}\left( \hat{u} \right) -1 \right] >0,\\
		Tr\left( J_{\hat{E}} \right) =-\delta \left[ H_{1}^{'}\left( \hat{v} \right) +H_{2}^{'}\left( \hat{u} \right) \right] <0.
	\end{aligned}
\end{equation*}
Therefore $\hat{E}\left( \hat{u},\ \hat{v} \right)$ is locally asymptotically stable. Considering again the Dulac function $g\left( u,\ v \right) =\frac{1}{u^2v^2}$. Applying $\delta \le \frac{s-e}{2}$, we get
\begin{equation*}
	\begin{aligned}
		\frac{\partial \left( gQ_1 \right)}{\partial u}+\frac{\partial \left( gQ_2 \right)}{\partial v}=\frac{e+2\delta -s}{u^2v^2}-\bar{M}<0,
	\end{aligned}
\end{equation*}
where $\bar{M}=\frac{1}{\left( m+u \right) ^2v^2}+\delta \left( \frac{1}{u^3v}+\frac{1}{uv^3} \right)>0$. Using the Bendixson-Dulac discriminant in the same way as Theorem 2.7, we can see that $\hat{E}\left( \hat{u},\ \hat{v} \right)$ is globally asymptotically stable. Theorem 2.10 is proved.  
\end{proof}

\begin{remark}
In our manuscript, we focus on how dispersal can keep the species with Allee effect from being extinct.  In detail, From Theorem 2.7 and Remark 2.1, we conclude that for the model (1.2) with a nonlinear dispersal mechanism, large amplitude of nonlinear dispersal can prevent the species with strong Allee effect from going extinct. However, Theorem 2.9 states that if the diffusion between two patches is linear, the species in both patches may still go extinct when Allee constant is large, even if the dispersal intensity is large. Theorem 2.10 states that if the diffusion between two patches is linear, the species in both patches can persist when Allee constant is large and the dispersal intensity is less. Through the comparison between Theorem 2,7, Remark 2.1 and Theorem 2.9, 2.10, it is not difficult to obtain that linear and nonlinear dispersal has different impact on the species' permanence. In all, large amplitude of nonlinear dispersal or less intensity of linear dispersal can keep the species with strong Allee effect from being extinct.  The above comparison of nonlinear diffusion with linear diffusion has theoretical and practical significance.
\end{remark}

\section{The PDE Case}

In this section we will study the effect of dispersal rate $\delta$  for spatially explicit PDE version of system (\ref{1.2}) and (\ref{2.10}).
\subsection{Notations and preliminary observations}

\begin{lemma}\label{lem:class1}
	Consider a $m\times m$ system of reaction-diffusion equations , where each equation is defined as follows: for all $i=1,...,m,$ 
	\begin{equation}
		\label{eq:class1}
		\partial_t u_i - d_i\Delta u_i = f_i(u_1,...,u_m)~\text{in}~ \mathbb{R}_+\times \Omega,~ \partial_\nu u_i = 0~ \text{on}~ \partial \Omega,~ u_i(0) = u_{i0},
	\end{equation}
	where $d_i \in (0,+\infty)$, $f=(f_1,...,f_m):\mathbb{R}^m \rightarrow \mathbb{R}^m$ is continuously differentiable on $\Omega$, and $u_{i0}\in L^{\infty}(\Omega)$. Then, there exists a time interval $T>0$ within which a unique classical solution to (\ref{eq:class1}) exists, i.e., the solution is well-defined and smooth on $[0,T)$. Let $T^*$ be the maximum value of all such intervals $T$. It follows that 
	\begin{equation*}
		\Bigg[\sup_{t \in [0,T^*), 1\leq i\leq m} ||u_i(t)||_{L^{\infty}(\Omega)} < +\infty \Bigg] \implies [T^*=+\infty].
	\end{equation*}
	If the non-linearity $(f_i)_{1\leq i\leq m}$ is additionally quasi-positive, which means 
	$$\forall i=1,..., m,~~\forall u_1,..., u_m \geq 0,~~f_i(u_1,...,u_{i-1}, 0, u_{i+1}, ..., u_m)\geq 0,$$
	then $$[\forall i=1,..., m, u_{i0}\geq 0] \implies [\forall i=1,...,m,~ \forall t\in [0,T^*), u_i(t)\geq 0].$$
\end{lemma}

\begin{lemma}\label{lem:class2}
	Under the same notations and assumptions as in Lemma \ref{lem:class1}, let's consider an additional condition. Suppose that $f$ exhibits at most polynomial growth and there exists $\mathbf{b}\in \mathbb{R}^m$ and a lower triangular invertible matrix $P$ with nonnegative entries, such that for any $r \in [0,+\infty)^m,$ we have $$Pf(r)\leq \Bigg[1+ \sum_{i=1}^{m} r_i \Bigg]\mathbf{b}.$$ Then, for any initial value $u_0 \in L^{\infty}(\Omega, \mathbb{R}_+^m),$ the system (\ref{eq:class1}) admits a strong global solution.
\end{lemma}

Based on the given assumptions, it is widely recognized that the following local existence result, originally presented by D. Henry in \cite{ref30}, holds true:

\begin{theorem}
	\label{thm:class3}
	The system (\ref{eq:class1}) possesses a unique and classical solution $(u,v)$ defined over the interval $[0,T_{\max}]\times \Omega$. If $T_{\max} < \infty$, then \begin{equation}
		\underset{t\nearrow T_{\max }}{\lim }\Big\{ \left\Vert u(t,.)\right\Vert
		_{\infty }+\left\Vert v(t,.)\right\Vert _{\infty } \Big\} =\infty ,  
	\end{equation}%
	where $T_{\max }$ denotes the eventual blow-up time in $\mathbb{L}^{\infty }(\Omega ).$
\end{theorem}

\subsection{A Case of Linear dispersal}

Consider the following spatially explicit PDE version of linear dispersal system motivated by ODE system $\eqref{2.10}$, resulting in the following reaction diffusion system, defined on $\Omega = [0,L]$,

\begin{equation}\label{PDE_linear}
			\left\{\begin{array}{l}
				\dfrac{\partial u}{\partial t}=\delta_{1}  u_{xx} + s_{1}(x)u\left( \dfrac{u}{m(x)+u}-e(x)-h(x)u \right),\vspace{2ex}\\
				\dfrac{\partial v}{\partial t}=\delta_{2}  v_{xx}+ s(x)v\left( \dfrac{v}{m_{1}(x)+v}-e_{1}(x)-h_{1}(x)v \right), \vspace{2ex}\\
				\dfrac{\partial u}{\partial \nu} = \dfrac{\partial v}{\partial \nu} =0, \quad \text{on} \quad \partial \Omega. \vspace{2ex}\\
				u(x,0)=u_0(x)>0, \quad v(x,0)=v_0(x) >0
			\end{array}\right.
		\end{equation}
Here

\begin{equation*}
m_1(x)=
\begin{cases}
	m(x), x \in [0,L_{1}],\\
	m(x)=0, x \in [L_{1}, L]
\end{cases}, \quad  e_1(x)=
\begin{cases}
e(x), x \in [0,L_{1}],\\
e(x)=0, x \in [L_{1}, L].
\end{cases} 
\end{equation*}

\begin{equation*}
	h_1(x)=
	\begin{cases}
		h(x), x \in [0,L_{1}],\\
		h(x)=1, x \in [L_{1}, L].
	\end{cases}, \quad  s_1 (x)=
	\begin{cases}
		s(x)=1, x \in [0,L_{1}],\\
		s(x), x \in [L_{1}, L].
	\end{cases} 
\end{equation*}


In this framework the patch structure is in a simple one dimensional domain $[0,L]$, where the region from $[0,L_{1}]$ is where the population is subject to an Allee effect, and the region from $[L_{1},L]$ is where the population is not subject to an Allee effect. Here linear dispersal is assumed for the populations modeled by the standard laplacian operator. 

One can typically think of the species $u$ as the population starting in the $[0,L_{1}]$ patch, where it is subject to an Allee effect, and will move via linear dispersal into the $[L_{1},L]$ patch. Once it enters this patch, it is not subject to an Allee effect anymore. Similarly we can think of the species $v$ as the population starting in the $[L_{1},L]$ patch, where there is no Allee effect in place. However it moves into the $[0,L_{1}]$ patch via linear dispersal, and upon entering this patch, it is immediately subject to an Allee effect. We consider the problem in spatial dimension $n=1$. Also the above mentioned functions $m(x),m_{1}(x),h(x),h_{1}(x),s(x),s_{1}(x),e(x),e_{1}(x)$ are all assumed to be in $L^{\infty}[0,L]$.
We can state the following result,

\begin{lemma}
	\label{lem:leml1}
	Consider the reaction diffusion system \eqref{PDE_linear}, then there exist  global in time non-negative classical solutions to this system, for certain positive bounded initial data.
\end{lemma}

\begin{proof}
Non-negativity of solutions follows via the quasi-positivity of the RHS of 
\eqref{PDE_linear}. Next via simple comparison for the $u$ equation we have,

\begin{equation}
	u\left( \dfrac{u}{m+u}-e-hu \right) \leq u\left( \dfrac{u}{u}-e-hu \right) = u\left( 1-e-hu \right)
\end{equation}
This follows using the positivity of the parameter $m$. Comparison with the logistic equation, via the use of Lemma \ref{lem:class2} yields the result. The analysis for the $v$ equation follows similarly.
\end{proof}

		\begin{figure}[h]
		\centering
		\subfigure[]{
			\includegraphics[width=0.4\linewidth]{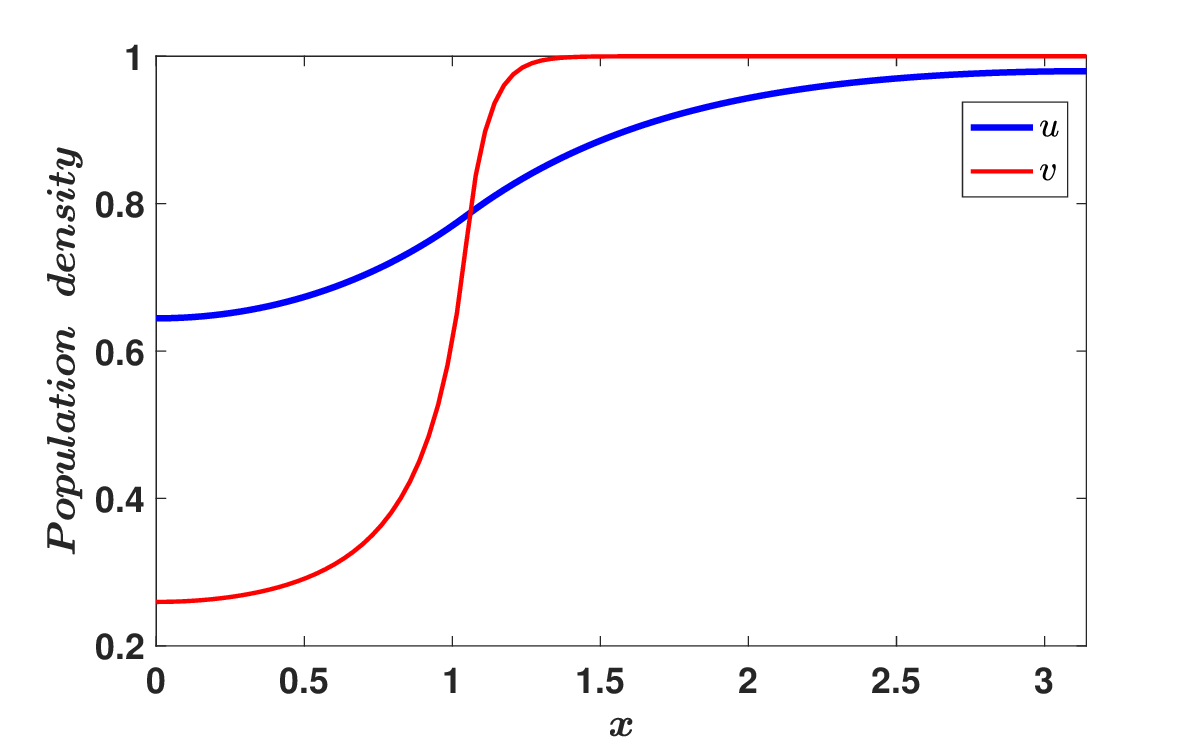}}
		\subfigure[]{
			\includegraphics[width=0.4\linewidth]{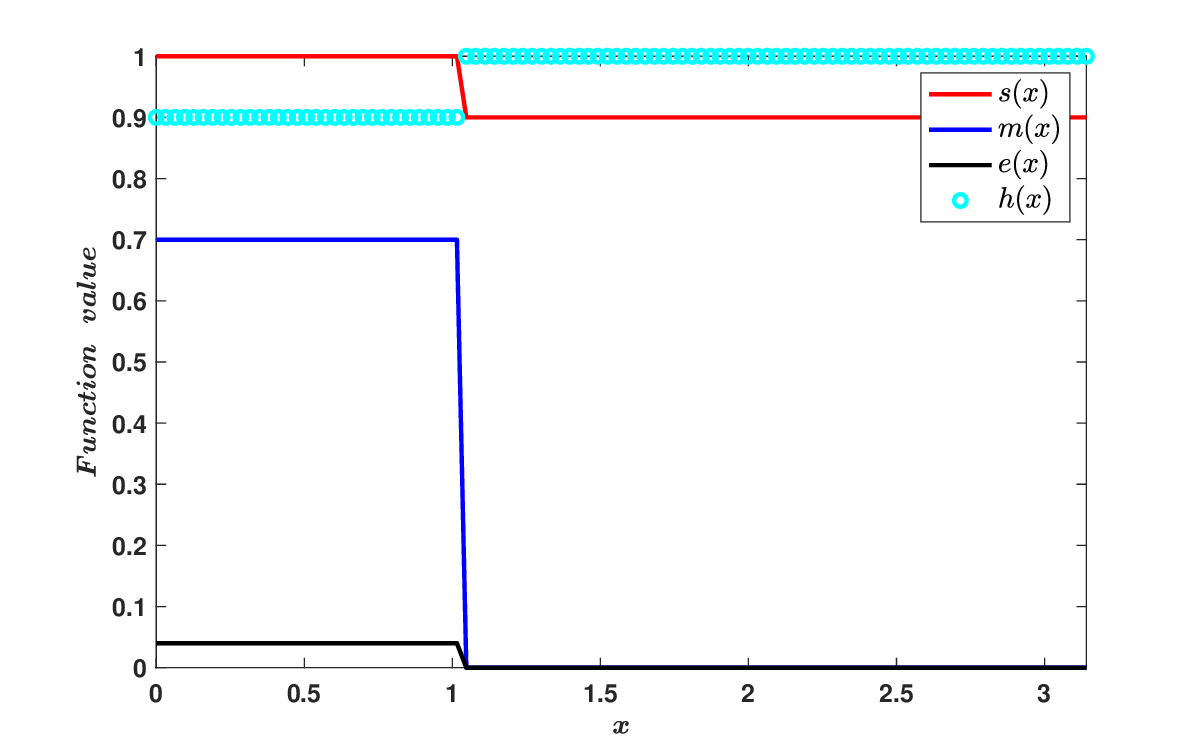}}
		\subfigure[]{
			\includegraphics[width=0.4\linewidth]{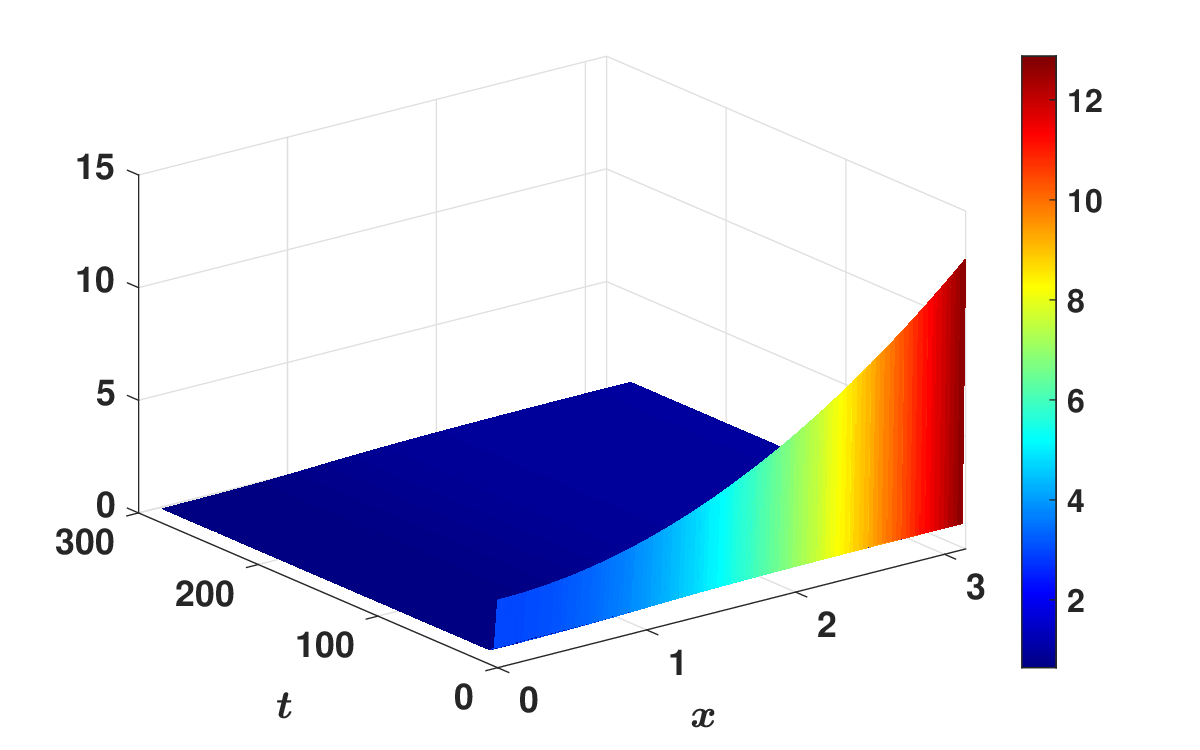}}
		\subfigure[]{
			\includegraphics[width=0.4\linewidth]{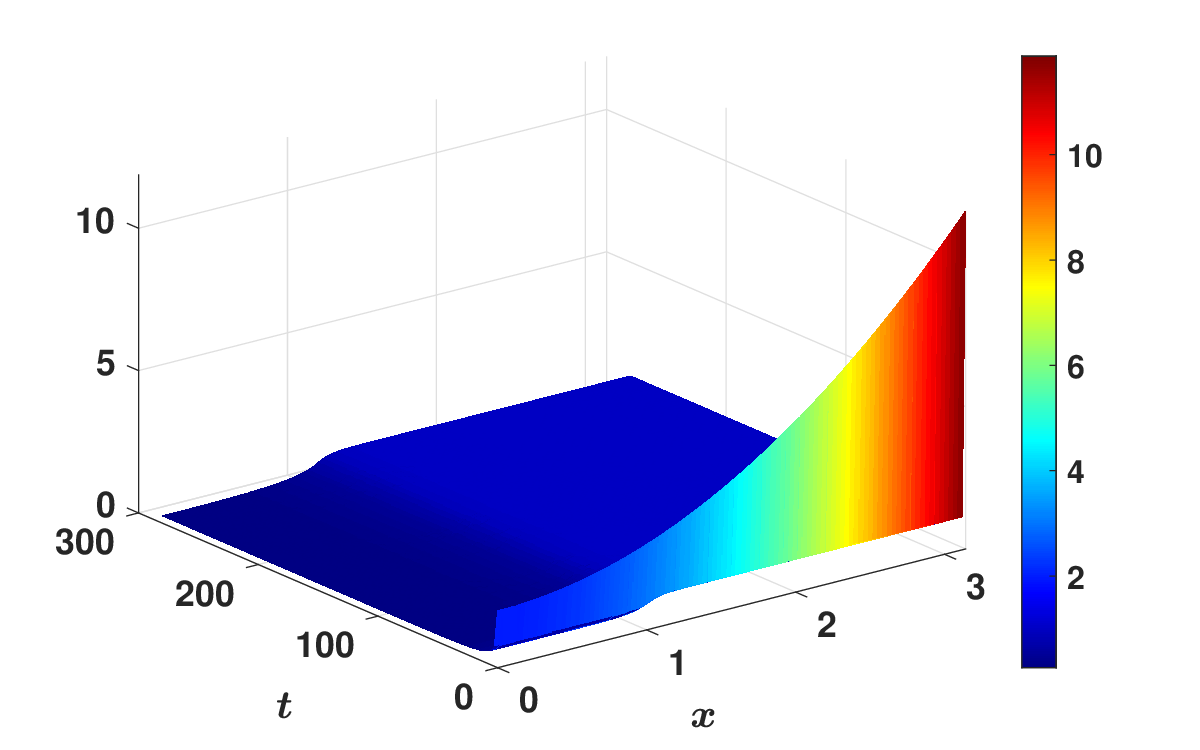}}
		\caption{Numerical simulation  illustrating the impact of a small linear dispersal parameter ($\delta_1 = 0.4,\delta_2 = 0.004$) on the dynamics of system \eqref{PDE_linear} in $\Omega=[0,\pi]$ for intial data $[u_0(x),v_0(x)]=[3+x^2,2+x^2].$ (a) Population density distribution vs space (b) Functional reponses used for simulation (c) Surface plot of $u$ (d) Surface plot of $v$.}
		\label{fig:lin_new}
	\end{figure}
	
	\begin{figure}[h]
		\centering
		\subfigure[]{
			\includegraphics[width=0.4\linewidth]{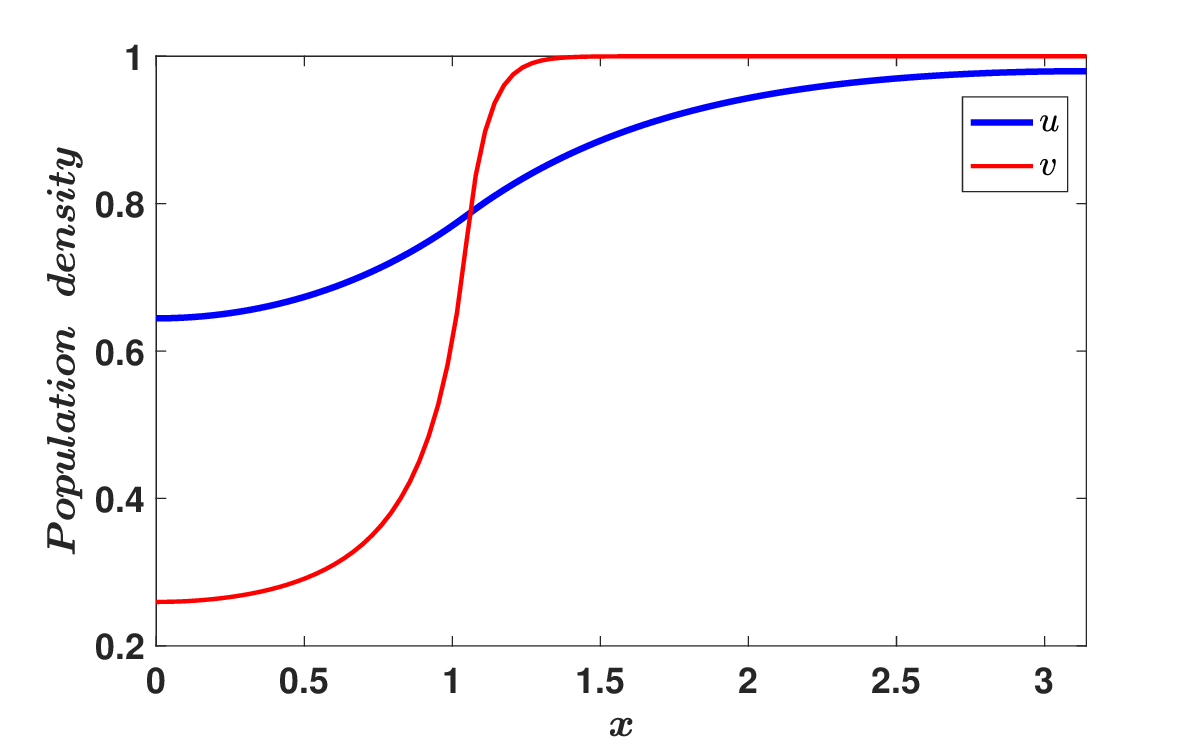}}
		\subfigure[]{
			\includegraphics[width=0.4\linewidth]{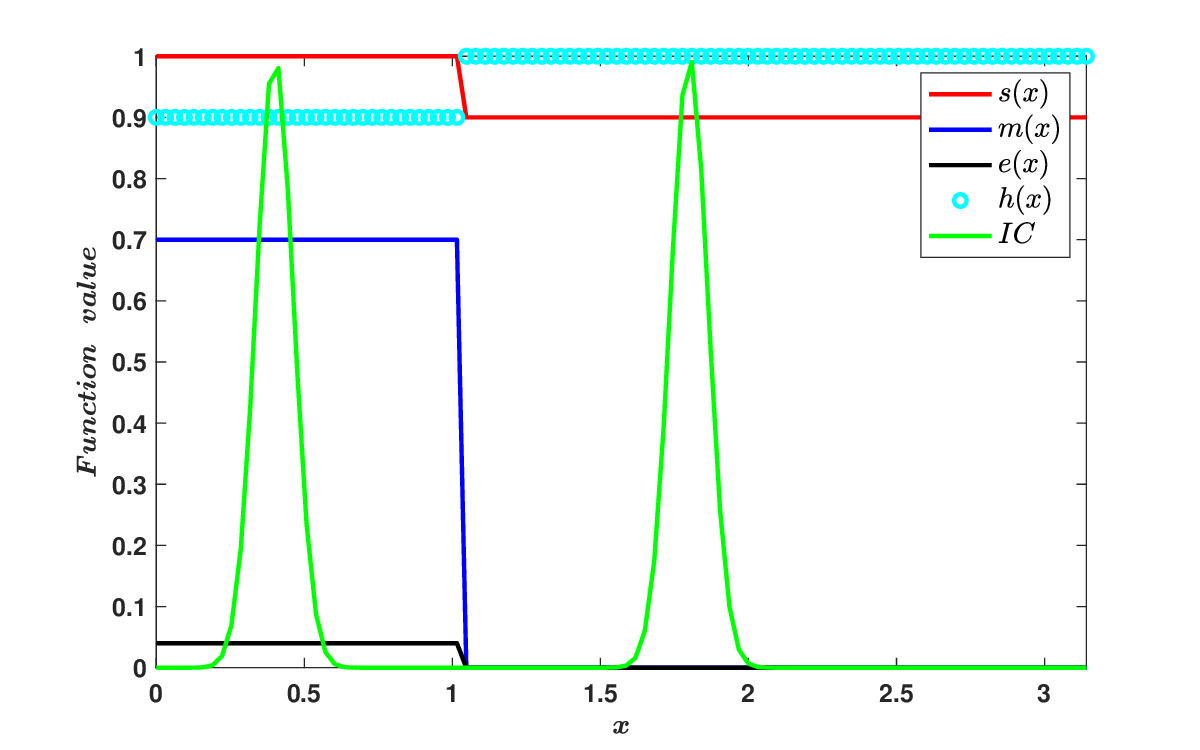}}
		\caption{Numerical simulation  illustrating the impact of a small linear dispersal parameter ($\delta_1 = 0.4,\delta_2 = 0.004$) on the dynamics of system \eqref{PDE_linear} in $\Omega=[0,\pi]$ for with non-flat intial data $[u_0(x),v_0(x)]=[e^{-(\frac{x-1.8}{\sqrt{.008}})^2}+e^{-(\frac{x-0.4}{\sqrt{.008}})^2},e^{-(\frac{x-1.8}{\sqrt{.008}})^2}+e^{-(\frac{x-0.4}{\sqrt{.008}})^2}].$ (a) Population density distribution vs space (b) Functional reponses used for simulation.}
		\label{fig:lin_peak_data1}
	\end{figure}
	
	\begin{figure}[h]
		\centering
		\subfigure[]{
			\includegraphics[width=0.4\linewidth]{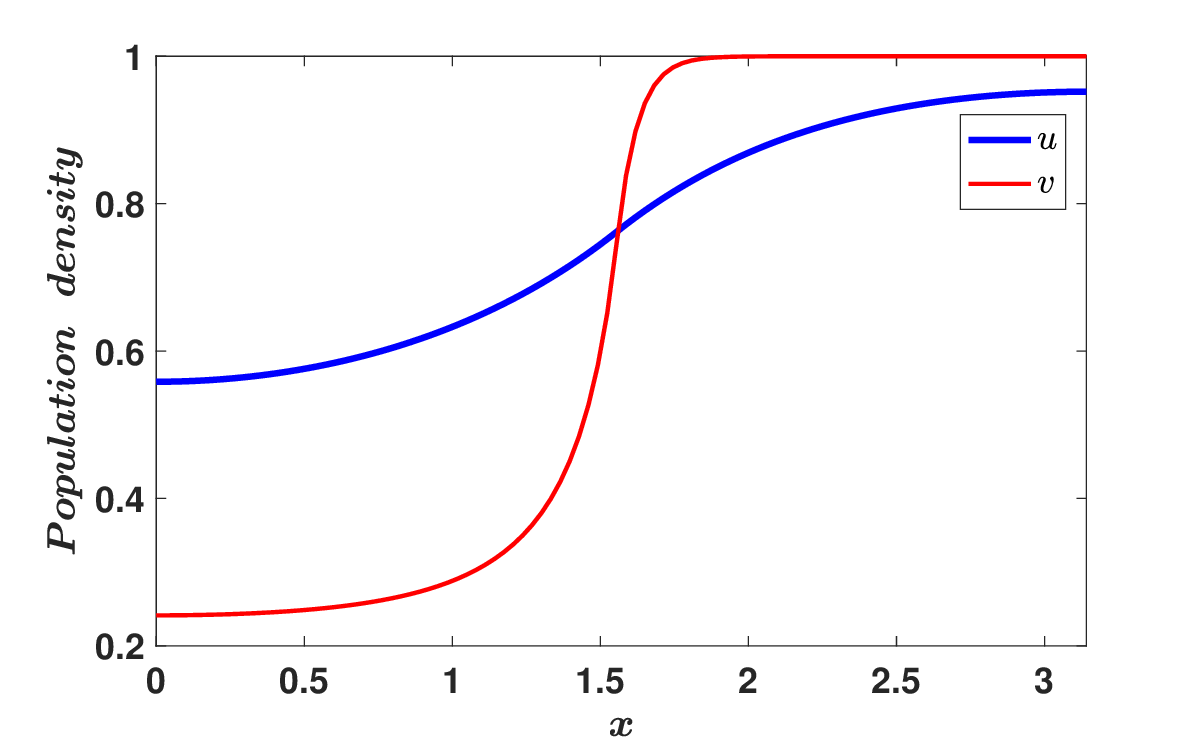}}
		\subfigure[]{
			\includegraphics[width=0.4\linewidth]{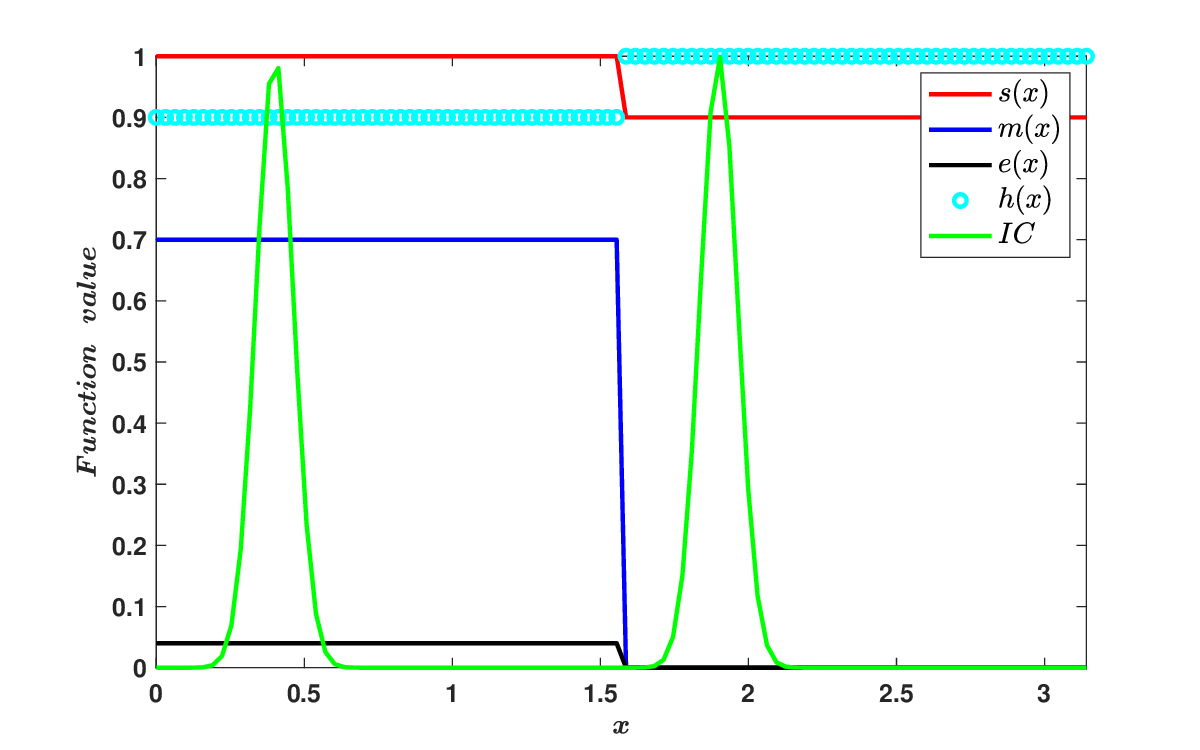}}
		\caption{Numerical simulation  illustrating the impact of a small linear dispersal parameter ($\delta_1 = 0.4,\delta_2 = 0.004$) on the dynamics of system \eqref{PDE_linear} in $\Omega=[0,\pi]$ for with non-flat intial data $[u_0(x),v_0(x)]=[e^{-(\frac{x-1.9}{\sqrt{.008}})^2}+e^{-(\frac{x-0.4}{\sqrt{.008}})^2},e^{-(\frac{x-1.9}{\sqrt{.008}})^2}+e^{-(\frac{x-0.4}{\sqrt{.008}})^2}].$ (a) Population density distribution vs space  (b) Functional reponses used for simulation.}
		\label{fig:lin_peak_data222}
	\end{figure}
	
	\begin{figure}[h]
		\centering
		\subfigure[]{
			\includegraphics[width=0.4\linewidth]{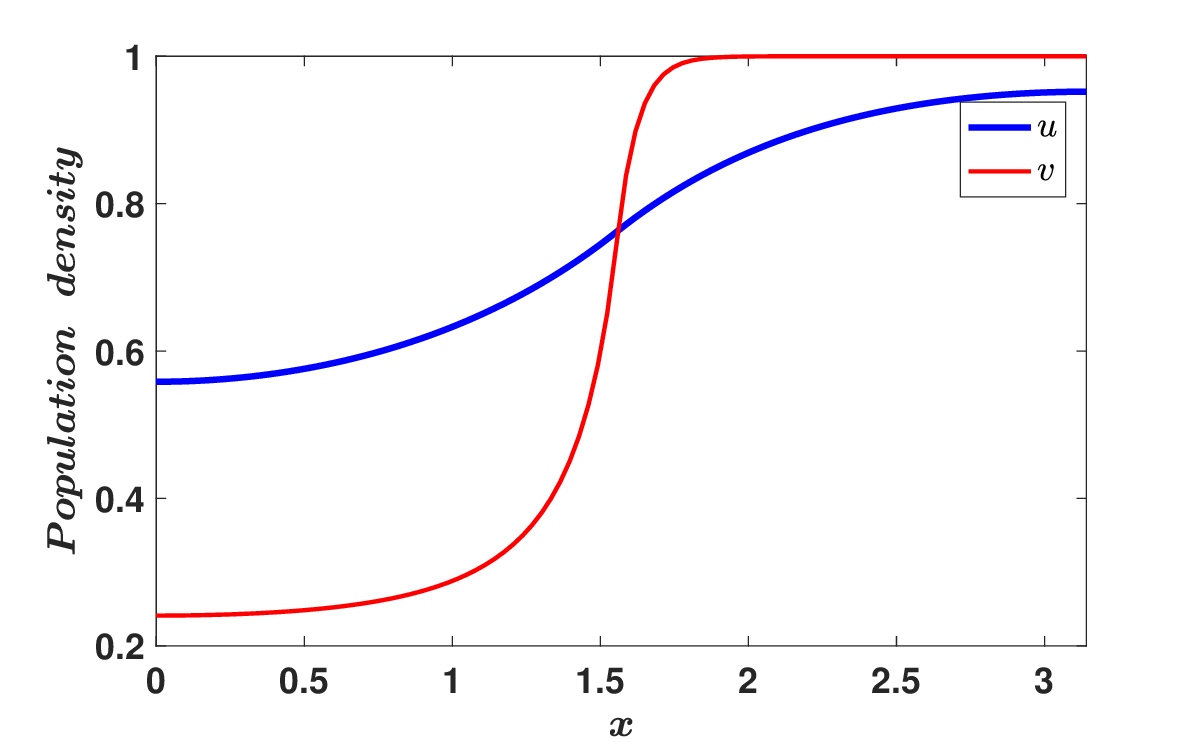}}
		\subfigure[]{
			\includegraphics[width=0.4\linewidth]{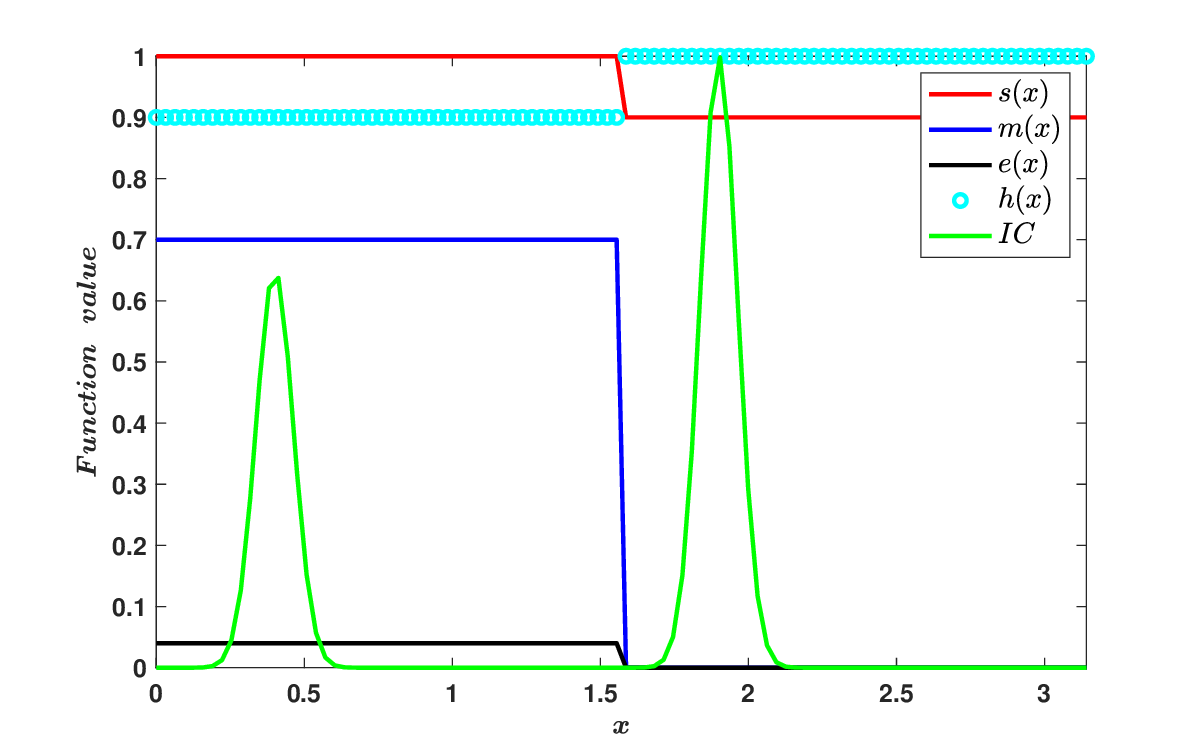}}
		\caption{Numerical simulation  illustrating the impact of a small linear dispersal parameter ($\delta_1 = 0.4,\delta_2 = 0.004$) on the dynamics of system \eqref{PDE_linear} in $\Omega=[0,\pi]$ for with non-flat intial data $[u_0(x),v_0(x)]=[e^{-(\frac{x-1.9}{\sqrt{.008}})^2}+0.65e^{-(\frac{x-0.4}{\sqrt{.008}})^2},e^{-(\frac{x-1.9}{\sqrt{.008}})^2}+0.65e^{-(\frac{x-0.4}{\sqrt{.008}})^2}].$ (a) Population density distribution vs space  (b) Functional reponses used for simulation.}
		\label{fig:lin_peak_data3}
	\end{figure}

	\subsection{A Case of Non-Linear dispersal}
	
	Consider the following spatially explicit version of non-linear dispersal system $\eqref{1.2}$, resulting in the following reaction diffusion system, defined on $\Omega = [0,L]$,

	
	\begin{equation}\label{PDE_nonlinear}
			\left\{\begin{array}{l}
				\dfrac{\partial u}{\partial t}=\delta_{1} u  u_{xx} + s_{1}(x)u\left( \dfrac{u}{m(x)+u}-e(x)-h(x)u \right),\vspace{2ex}\\
				\dfrac{\partial v}{\partial t}=\delta_{2} v  v_{xx}+ s(x)v\left( \dfrac{v}{m_{1}(x)+v}-e_{1}(x)-h_{1}(x)v \right), \vspace{2ex}\\
				\dfrac{\partial u}{\partial \nu} = \dfrac{\partial v}{\partial \nu} =0, \quad \text{on} \quad \partial \Omega. \vspace{2ex}\\
				u(x,0)=u_0(x)>0, \quad v(x,0)=v_0(x) >0
			\end{array}\right.
		\end{equation}

	\begin{equation*}
		m_1(x)=
		\begin{cases}
			m(x), x \in [0,L_{1}],\\
			m(x)=0, x \in [L_{1}, L]
		\end{cases}, \quad  e_1(x)=
		\begin{cases}
			e(x), x \in [0,L_{1}],\\
			e(x)=0, x \in [L_{1}, L].
		\end{cases} 
	\end{equation*}
	
	\begin{equation*}
		h_1(x)=
		\begin{cases}
			h(x), x \in [0,L_{1}],\\
			h(x)=1, x \in [L_{1}, L].
		\end{cases}, \quad  s_1 (x)=
		\begin{cases}
			s(x)=1, x \in [0,L_{1}],\\
			s(x), x \in [L_{1}, L].
		\end{cases} 
	\end{equation*}


In this framework the patch structure is in a simple one dimensional domain $[0,L]$, where the region from $[0,L_{1}]$ is where the population is subject to an Allee effect, and the region from $[L_{1},L]$ is where the population is not subject to an Allee effect. Here non-linear dispersal is assumed for the populations modeled by a non-standard laplacian operator. 
We consider the problem in spatial dimension $n=1$. Again the functions $m(x),m_{1}(x),h(x),h_{1}(x),s(x),s_{1}(x),e(x),e_{1}(x)$ are all assumed to be nonnegative functions in $L^{\infty}[0,L]$. Furthermore, since the $s(x) , h(x)$ functions have to mimic the $h,s$ parameters from the ODE systems considered earlier, we assume that there exists a positive constant $C_{1}$ s.t $0<C_{1} < \min(h(x),h_{1}(x),s(x),s_{1}(x))$.
	
We state the following result,
	
	\begin{theorem}
		Consider the reaction diffusion system \eqref{PDE_nonlinear}, then there exist  global in time non-negative classical solutions to this system, for certain positive bounded initial data.
	\end{theorem}

	\begin{proof}
	Consider the $u$ equation for the reaction diffusion system \eqref{PDE_nonlinear}. Dividing through by $u$ we obtain,
	the following equivalent equation,

	\begin{equation}\label{PDE_nonlinear_equiv}
		\left\{\begin{array}{l}
			\dfrac{\partial }{\partial t}\left( \log u\right)=\delta_{1}  u_{xx} + s_{1}(x)\left( \dfrac{u}{m(x)+u}-e(x)-h(x)u \right),\vspace{2ex}\\
			
		\end{array}\right.
	\end{equation}
	This follows by formally dividing through by $u, v$ assuming positivity. Integrating the above equation over $\Omega$ yields,
	
	\begin{equation*}
		\frac{d}{dt}\int_{\Omega}\log(u) dx + \int_{\Omega} s_{1}(x) h(x)  u dx \leq s_{1}(x)|(1-e(x))||\Omega|.
	\end{equation*}
	
	from which it follows that,
	
	\begin{equation*}
		\frac{d}{dt}\int_{\Omega}\log(u) dx + (C_{1})^{2}\int_{\Omega} u dx \leq ||s_{1}(x)||_{\infty}||(1-e(x))||_{\infty} |\Omega|.
	\end{equation*}
	
This follows using the earlier estimate on the RHS of $u$ equation in Lemma \ref{lem:leml1}, the assumption that $0<C_{1} < \min(h(x),h_{1}(x),s(x),s_{1}(x))$. 
	Now using the inequality $\log(x) < x, x>0$, we obtain
	
	\begin{equation*}
		\frac{d}{dt}\int_{\Omega}\log(u) dx + (C_{1})^{2} \int_{\Omega} \log(u) dx \leq C_{2}|\Omega|.
	\end{equation*}
	An application of Gronwall inequality yields,
	\begin{equation*}
		\int_{\Omega}\log(u) dx  \leq \frac{C_{2}|\Omega|}{C_{1}^{2}} + \log(u_{0}(x))
	\end{equation*}

%
%
	
	Similar analysis follows for the $v$ equation. Thus the $L^{1}(\Omega)$ norms of the $\log(u)$ cannot blow-up at any finite time $T^{*} < \infty$, for suitable initial data $u_{0}(x)$ s.t. the 
	$\log(u_{0}(x))$ is well defined. This in turn yields control of the $L^{1}(\Omega)$ norms of the solution. Here $C_{2}$ is a pure constants that could absorb the other parameters in the problem. This, in conjunction with classical theory \cite{ref30}, where essentially one needs to control the RHS of \eqref{PDE_nonlinear}, in $L^{p}$ for $p>\frac{n}{2}$, yields the result. 
	\end{proof}

\begin{figure}[h]
	\centering
	\subfigure[]{
		\includegraphics[width=0.4\linewidth]{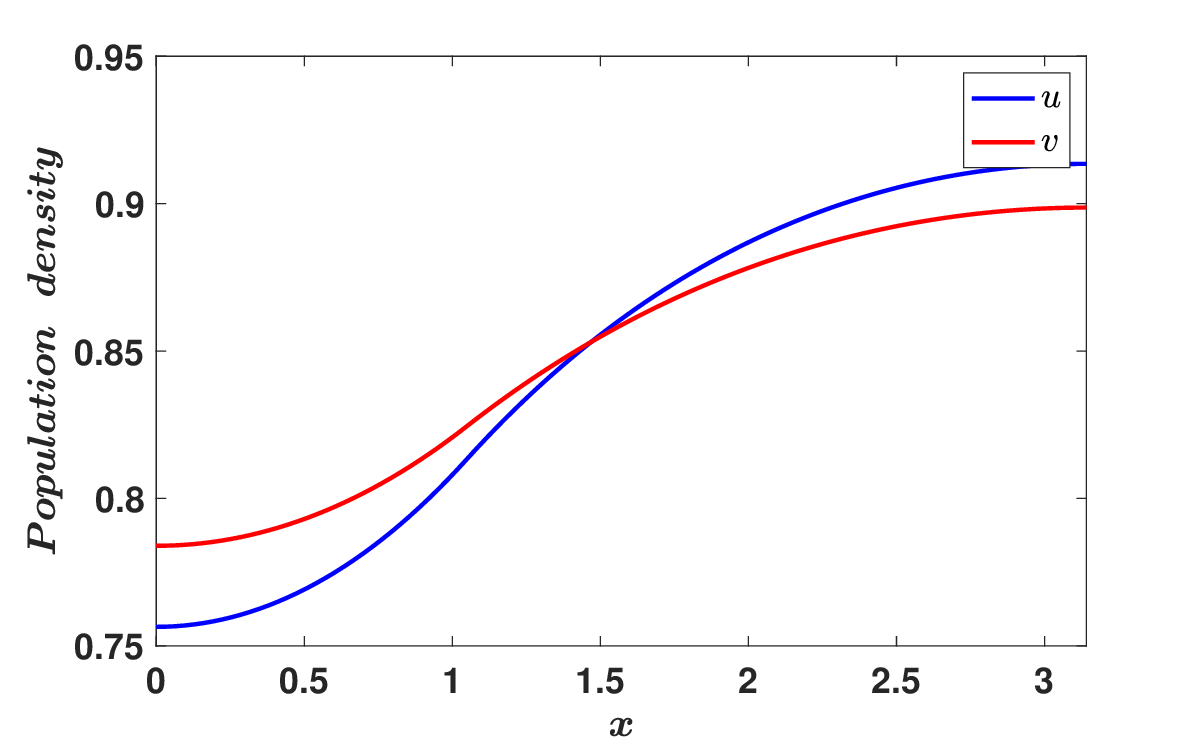}}
	\subfigure[]{
		\includegraphics[width=0.4\linewidth]{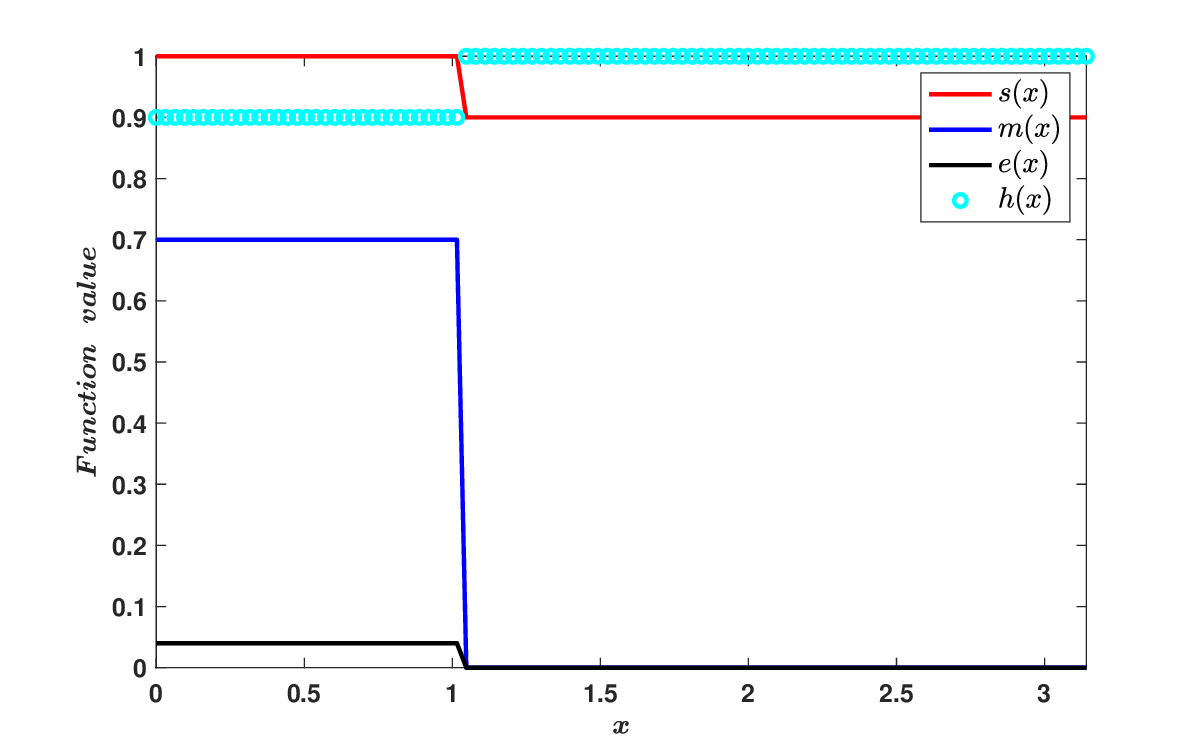}}
	\subfigure[]{
		\includegraphics[width=0.4\linewidth]{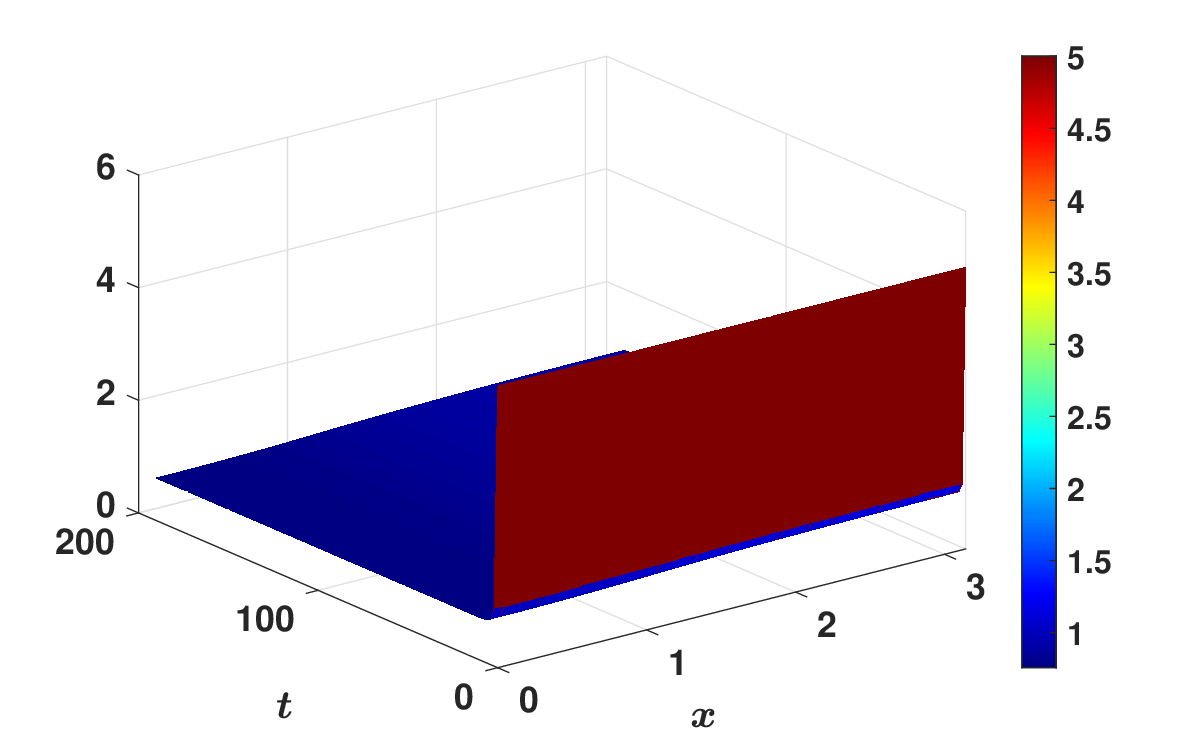}}
	\subfigure[]{
		\includegraphics[width=0.4\linewidth]{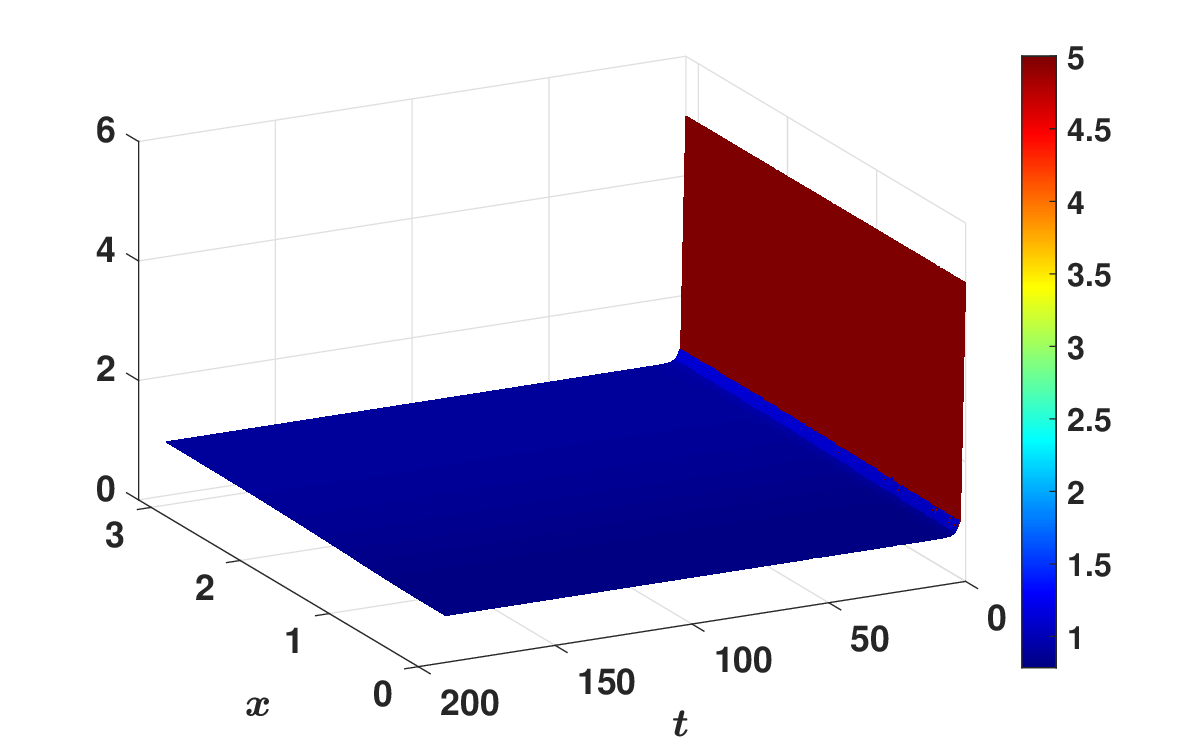}}
	\caption{Numerical simulation  illustrating the impact of a large non-linear dispersal parameter ($\delta_1 = 2,\delta_2 = 3$) on the dynamics of system \eqref{PDE_nonlinear} in $\Omega=[0,\pi]$ for intial data $[u_0(x),v_0(x)]=[5,5].$ (a) Population density distribution vs space (b) Functional reponses used for simulation (c) Surface plot of $u$ (d) Surface plot of $v$.}
	\label{fig:nonlin_new}
\end{figure}

\begin{figure}[h]
	\centering
	\subfigure[]{
		\includegraphics[width=0.4\linewidth]{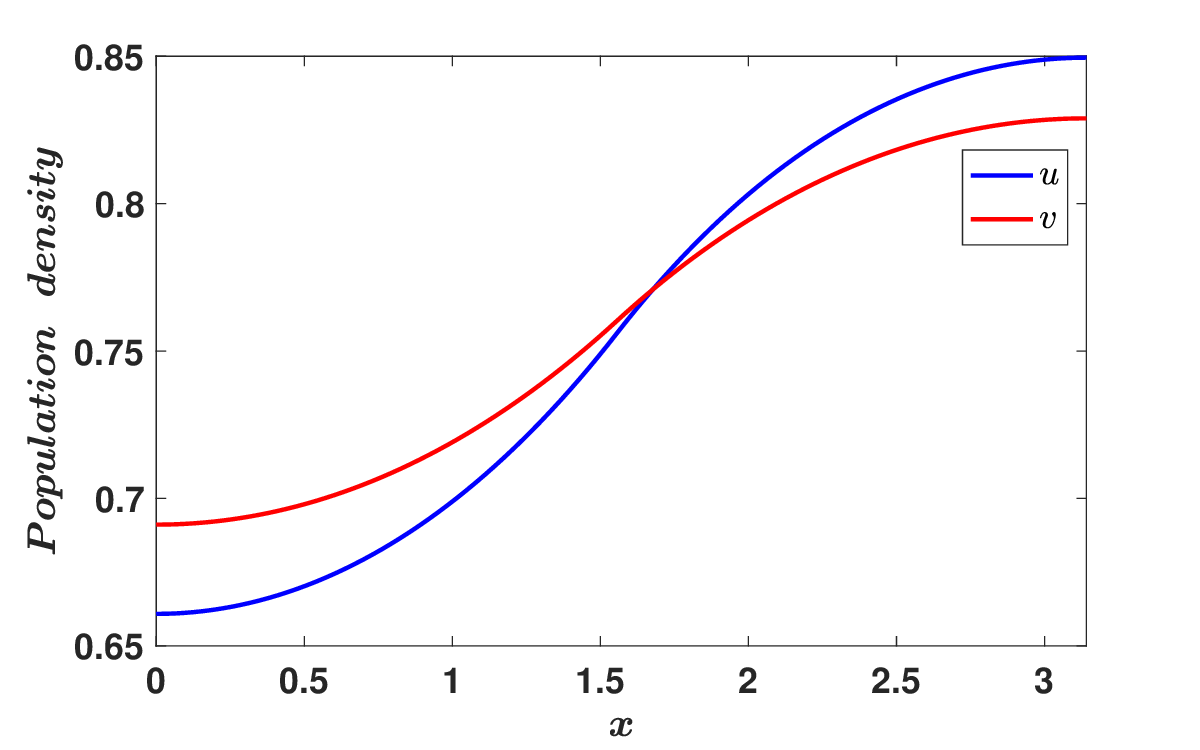}}
	\subfigure[]{
		\includegraphics[width=0.4\linewidth]{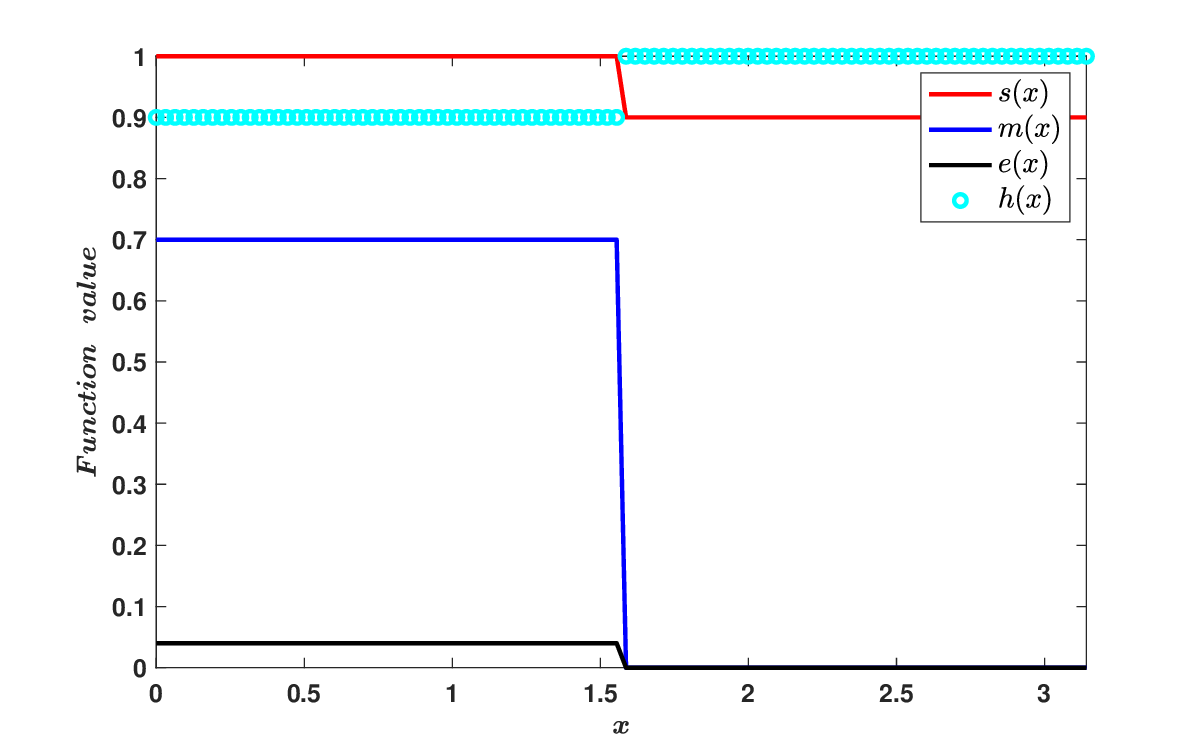}}
	\caption{Numerical simulation  illustrating the impact of a large non-linear dispersal parameter ($\delta_1 = 2,\delta_2 = 3$) on the dynamics of system \eqref{PDE_nonlinear} in $\Omega=[0,\pi]$ for with non-flat intial data $[u_0(x),v_0(x)]=[5,5].$ (a) Population density distribution vs space (b) Functional reponses used for simulation.}
	\label{fig:nonlin_peak_data1}
\end{figure}

\begin{figure}[h]
	\centering
	\subfigure[]{
		\includegraphics[width=0.4\linewidth]{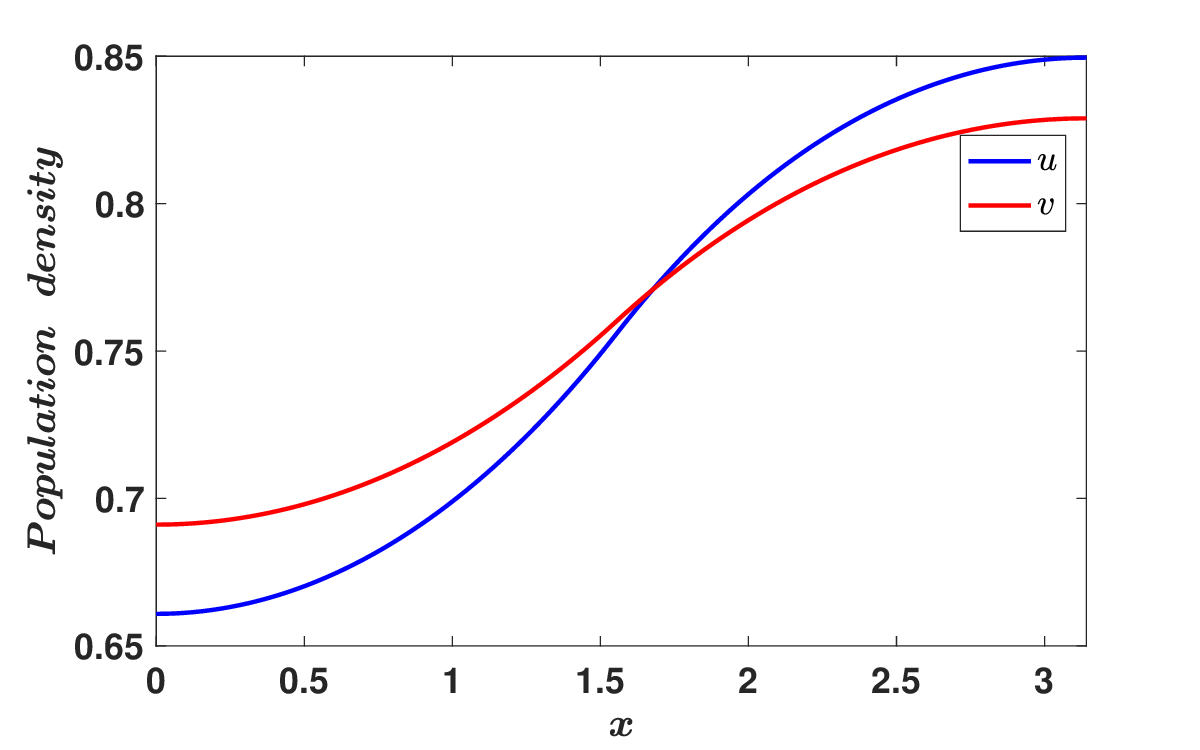}}
	\subfigure[]{
		\includegraphics[width=0.4\linewidth]{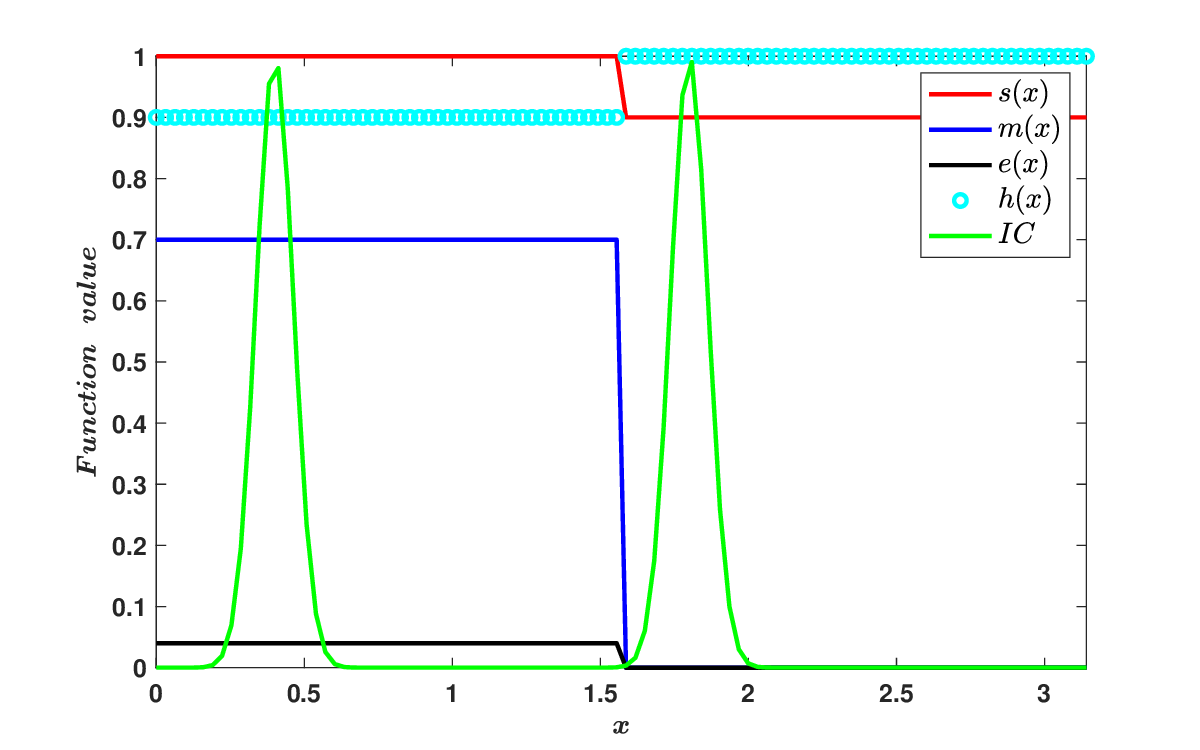}}
	\caption{Numerical simulation  illustrating the impact of a large non-linear dispersal parameter ($\delta_1 = 2,\delta_2 = 3$) on the dynamics of system \eqref{PDE_nonlinear} in $\Omega=[0,\pi]$ for with non-flat intial data $[u_0(x),v_0(x)]=[0.0001+e^{-(\frac{x-1.9}{\sqrt{.008}})^2}+e^{-(\frac{x-0.4}{\sqrt{.008}})^2},0.0001+e^{-(\frac{x-1.9}{\sqrt{.008}})^2}+e^{-(\frac{x-0.4}{\sqrt{.008}})^2}].$ (a) Population density distribution vs space  (b) Functional reponses used for simulation.}
	\label{fig:nonlin_peak_data2}
\end{figure}

\subsection{Numerical Simulations}
The numerical simulations for both linear equation \eqref{PDE_linear} and nonlinear equation \eqref{PDE_nonlinear} in the context of dispersal partial differential equations (PDEs) were executed using MATLAB R2021b. The simulations employed the built-in function \verb|pdepe|, specifically designed to solve one-dimensional parabolic and elliptic PDEs. The spatial domain was set as the unit-sized interval $[0,1]$, which was discretized into $100$ sub-intervals. It is numerically validated under some parametric restriction, and for some given data, a large magnitude of nonlinear dispersal or less intensity of linear dispersal may prevent species with the Allee effect from becoming extinct (See Figs~[	\ref{fig:lin_new},\ref{fig:lin_peak_data1},\ref{fig:lin_peak_data222},\ref{fig:lin_peak_data3},\ref{fig:nonlin_new},\ref{fig:nonlin_peak_data1},\ref{fig:nonlin_peak_data2}
].

		\section{Conclusions}
		In this paper, the interplay of the Allee effect and nonlinear dispersal on a two-patch model are studied. Our goal is to demonstrate whether nonlinear diffusion between the two patches contributes to overcoming the Allee effect.
		
		When the dispersal intensity is low, we have concluded that population $u$  will go extinct when $B<e<1, m>m^*$. Besides, we also proved that the two positive equilibrium points $E_1\left( u_1,\ v_1 \right)$ and $E_2\left( u_2,\ v_2 \right)$ of model (\ref{1.2}) will undergo saddle-node bifurcation when $m=m^*$. These findings suggest that the Allee effect has a major impact on the extinction of the population in patch 1 when the dispersal intensity is very weak. However, the positive equilibrium $E_1\left( u_1,\ v_1 \right)$ is always globally asymptotically stable when $e<B$, i.e., $\delta>\frac{se}{s-e}$. The above result shows that both species will persist when the nonlinear dispersal intensity is high. In other words, under large nonlinear dispersal, the persistence of both species seems independent of the Allee effect. 
		
		Besides, for the corresponding model with linear dispersal, we have obtained two interesting  results when the Allee effect is large. The results show that when the linear dispersal is high, both the species will go extinct. However, the species with low linear dispersal will persist.  We declare that a large magnitude of nonlinear dispersal or a less intensity of linear dispersal  may prevent species with a strong Allee effect from becoming extinct. 
		
		The above are results derived in the ODE case. We derive analogous results in the PDE case as well. Herein, we set up a one dimensional domain to have two explicit patches, one patch where the populations are subject to an Allee effect and the other patch where they are not. The species move in and out of these patches via linear diffusion, as well as non-linear diffusion. What we observe in the PDE case is that high nonlinear diffusion can eliminate the effects of the strong Allee effect so that populations do not become extinct, see Fig. \ref{fig:nonlin_new}-\ref{fig:nonlin_peak_data2}. This complements the ODE model's findings with continuous patches models. We also observe is that in the case of linear diffusion, small intensity of diffusion, can also lead to coexistence. This again is in accordance with our ODE findings. This is also true, even if we take an initial condition in the patch subject to an Allee effect, s.t. the $||u_{0}||_{\infty} < M$, that is the peak of the initial data is under the Allee threshold. Thus the small linear diffusion allows the species to disperse into the second patch, and escape the Allee effect, before it can cause local extinction, see Fig. \ref{fig:lin_peak_data3}. 
		
		Another important use of the PDE model is applications to habitat fragmentation due to human intervention and exploitation. Some fragmented patches are large while others are small in area. Although ODE models with patch structure are powerful, but all in all are not spatially explicit - so cannot capture this effect of patches of different sizes. They cannot explicitly model the case of a patch, which is actually say 10, 20 or 30 percent of the entire domain. However, in the PDE case with patch structure we can model this situation. We include this in our simulations in the PDE case. You will see that we try patches which are both 1/3 the size of the entire domain as well as 1/2 the size of the domain. We note that in both cases we can obtain coexistence. Thus patch size does not seem to play a part in achieving the coexistence dynamics, one can see this by comparing Fig. \ref{fig:nonlin_new} to Fig. \ref{fig:nonlin_peak_data2} or Fig. \ref{fig:lin_peak_data1} to Fig. \ref{fig:lin_new} . Proving this rigorously will make for very interesting future work.

            Recent studies by Srivastava et al. \cite{ref31} and Chen et al. \cite{ref32} have inspired further exploration. These studies investigate the impact of fear in a purely competitive two-species model, where one species instills fear in the other. To advance our understanding, it would be valuable to enhance future research by extending the incorporation of the fear effect into both linear and nonlinear dispersal systems, considering both ordinary differential equation (ODE) and partial differential equation (PDE) formulations. In summary, future investigations that modify the ODE and PDE versions of linear and nonlinear dispersal systems to incorporate the fear effect build upon recent work and have the potential to deepen our understanding of ecological dynamics. By exploring the role of fear in species interactions, we can uncover new dimensions and pave the way for more accurate modeling and conservation approaches. Other future directions in the PDE cases will include investigating edge effects with patch structure, \cite{m20} and investigating blow-up prevention with patches \cite{P16, v23}.

		As is known, the traditional growth rate of species is logistic and an Allee effect in place may cause a species to become extinct. To this end it is interesting to propose a model with two patches, i.e., one where there is logistic growth and the other where there is an Allee effect. Thus, in this manuscript, We have explored suitable dispersal strategy which can benefit species in both patches. In other words, we discuss the impact of dispersal on keeping the species subject to an Allee effect from extinction. Our results shows that a large magnitude of nonlinear dispersal or a less intensity of linear dispersal may prevent species subject to a strong Allee effect from becoming extinct. We point out that it is also meaningful to investigate a two patch model, where the form of the Allee effect in the various patches could change - such as a strong effect in one patch versus a weak effect in the other patch. We leave such an investigation for future work.

		\section*{Acknowledgment}
		\hspace*{\parindent}This work was supported by the National
		Natural Science Foundation of China under
		Grant(11601085) and the Natural Science Foundation of Fujian
		Province(2021J01614, 2021J01613).

        \section*{Conflict of interest}
        The authors declare there is no conflict of interest.



\begin{thebibliography}{99}
\bibitem{ref36}
\newblock M.E. Soule, D. Simberloff,
 \newblock What do genetics and ecology tell us about the design of nature reserves?
 \newblock \emph{Biological Conservation}, \textbf{35(1)} (1986), 19--40.
\doilink{https://doi.org/10.1016/0006-3207(86)90025-X}


\bibitem{ref37}
\newblock R. Channell, M. Lomolino,
\newblock Dynamic biogeography and conservation of endangered species,
\newblock \emph{Nature}, \textbf{403} (2000), 84--86.
\doilink{https://doi.org/10.1038/47487}




\bibitem{ref1}
\newblock A. Mai, G. Sun, F. Zhang, L. Wang,
\newblock The joint impacts of dispersal delay and dispersal patterns on the stability of predator-prey metacommunities,
\newblock \emph{Journal of Theoretical Biology}, \textbf{462} (2019), 455--465.
\doilink{https://doi.org/10.1016/j.jtbi.2018.11.035}


\bibitem{ref2}
\newblock Y. Kang, S.K. Sasmal, K. Messan,
\newblock A two-patch prey-predator model with predator dispersal driven by the predation strength,
\newblock \emph{Mathematical Biosciences and Engineering}, \textbf{14(4)} (2017), 843--880.
\doilink{https://doi.org/10.3934/mbe.2017046}



\bibitem{ref5}
\newblock J. Ban, Y. Wang, H. Wu, 
\newblock Dynamics of predator-prey systems with prey’s dispersal between patches,
\newblock \emph{Indian Journal of Pure and Applied Mathematics}, \textbf{53} (2022), 550--569.
\doilink{https://doi.org/10.1007/s13226-021-00117-5}






\bibitem{ref6}
\newblock K. Hu, Y. Wang,
\newblock Dynamics of consumer-resource systems with consumer's dispersal between patches,
\newblock \emph{Discrete and Continuous Dynamical Systems - Series B}, \textbf{27(2)} (2022), 977--1000.
\doilink{https://doi.org/10.3934/dcdsb.2021077}


\bibitem{ref7}
\newblock Z. Wang, Y. Wang,
\newblock Bifurcations in diffusive predator–prey systems with Beddington–DeAngelis functional response,
\newblock \emph{Nonlinear Dynamics}, \textbf{105} (2021), 1045--1061.
\doilink{https://doi.org/10.1007/s11071-021-06635-5}


\bibitem{ref34}
\newblock L.J.S. Allen,
\newblock Persistence and extinction in single-species reaction-diffusion models,
\newblock \emph{Bulletin of Mathematical Biology}, \textbf{45(2)} (1983), 209--227.
\doilink{https://doi.org/10.1016/S0092-8240(83)80052-4}






\bibitem{ref35}
\newblock S.A. Levin, L.A. Segel,
\newblock Hypothesis for origin of planktonic patchiness,
\newblock \emph{Nature}, \textbf{259} (1976), 659.
\doilink{https://doi.org/10.1038/259659a0}

\bibitem{ref33}
\newblock W.S.C. Gurney, R.M. Nisbet,
\newblock The regulation of inhomogeneous populations,
\newblock \emph{Journal of Theoretical Biology}, \textbf{52(2)} (1975), 441--457.
\doilink{https://doi.org/10.1016/0022-5193(75)90011-9}

\bibitem{Keyghobadi2007} Keyghobadi, N. (2007). 
\newblock The genetic implications of habitat fragmentation for animals. 
\newblock Canadian Journal of Zoology, 85, 1049-1064.



\bibitem{Luo2022} Luo, M., Wang, S., Saavedra, S., Ebert, D. $\&$ Altermatt, F. (2022). 
\newblock Multispecies coexistence in fragmented landscapes. 
\newblock Proceedings of the National Academy of Sciences, 119, e2201503119.


\bibitem{Rohwader2022} Rohwäder, M.S. $\&$ Jeltsch, F. (2022). 
\newblock Foraging personalities modify effects of habitat fragmentation on biodiversity. 
\newblock Oikos, e09056.

\bibitem{Alan2002} Franklin, Alan B., Barry R. Noon, and T. Luke George. 
\newblock "What is habitat fragmentation?." 
\newblock Studies in avian biology 25 (2002): 20-29.


\bibitem{ref11}
\newblock X. Zhang, L. Chen,
\newblock The linear and nonlinear diffusion of the competitive Lotka–Volterra model,
\newblock \emph{Nonlinear Analysis: Theory, Methods and Applications}, \textbf{66(12)} (2007), 2767--2776.
\doilink{https://doi.org/10.1016/j.na.2006.04.006}








\bibitem{ref12}
\newblock X. Zhou, X. Shi, X. Song,
\newblock Analysis of nonautonomous predator-prey model with nonlinear diffusion and time delay,
\newblock \emph{Applied Mathematics and Computation,}, \textbf{196(1)} (2008), 129--136.
\doilink{https://doi.org/10.1016/j.amc.2007.05.041}








\bibitem{ref15}
\newblock Z. Zhu, Y. Chen, Z. Li, F. Chen,
\newblock Dynamic behaviors of a Leslie-Gower model with strong Allee effect and fear effect in prey,
\newblock \emph{Mathematical Biosciences and Engineering}, \textbf{20(6)} (2023), 10977--10999.
\doilink{https://doi.org/10.3934/mbe.2023486}



\bibitem{ref16}
\newblock Y. Liu, Z. Li, M. He,
\newblock Bifurcation analysis in a Holling-Tanner predator-prey model with strong Allee effect,
\emph{Mathematical Biosciences and Engineering}, \textbf{20(5)} (2023), 8632--8665.
\doilink{https://doi.org/10.3934/mbe.2023379}








\bibitem{Debinski2000} Debinski, D.M. $\&$ Holt, R.D. (2000).
\newblock A survey and overview of habitat fragmentation experiments.
\newblock Conservation biology, 14, 342-355.





\bibitem{Fahrig2002} Fahrig, L. (2002). 
\newblock Effect of habitat fragmentation on the extinction threshold: A synthesis. 
\newblock Ecol Appl, 12, 346-353.


\bibitem{Fahrig2017} Fahrig, L. (2017). 
\newblock Ecological responses to habitat fragmentation per se. 
\newblock Annual Review of Ecology, Evolution, and Systematics, 48, 1-23.


\bibitem{Fahrig2019} Fahrig, L. (2019). 
\newblock Habitat fragmentation: A long and tangled tale. Global Ecol 
\newblock Biogeogr, 28, 33-41.












\bibitem{ref22}
\newblock T. Liu, L. Chen, F. Chen, Z. Li,
\newblock Dynamics of a Leslie-Gower Model with weak Allee effect on prey and fear effect on predator,
\newblock \emph{International Journal of Bifurcation and Chaos}, \textbf{33(1)} (2023), 2350008.
\doilink{https://doi.org/10.1142/S0218127423500086}

 \bibitem{ref23}
 \newblock T. Liu, L. Chen, F. Chen, Z. Li,
\newblock Stability analysis of a Leslie-Gower model with strong Allee effect on prey and fear effect on predator,
\newblock \emph{International Journal of Bifurcation and Chaos}, \textbf{32(6)} (2022), 2250082.
 \doilink{https://doi.org/10.1142/S0218127422500821}
 
 \bibitem{ref21}
 \newblock Y. Lv, L. Chen, F. Chen, Z. Li,
 \newblock Stability and bifurcation in an SI epidemic model with additive Allee effect and time delay,
 \newblock \emph{International Journal of Bifurcation and Chaos}, \textbf{31(4)} (2021), 2150060.
 \doilink{https://doi.org/10.1142/S0218127421500607}


\bibitem{ref24}
\newblock L. Chen, T. Liu, F. Chen,
\newblock Stability and bifurcation in a two-patch model with additive Allee effect,
\newblock \emph{AIMS Mathematics}, \textbf{7(1)} (2022), 536--551.
\doilink{https://doi.org/10.3934/math.2022034}

\bibitem{ref25}
\newblock W. Wang, 
\newblock Population dispersal and Allee effect,
\newblock \emph{Ricerche di Matematica}, \textbf{65} (2016), 535--548.
\doilink{https://doi.org/10.1007/s11587-016-0273-0}


\bibitem{ref38}
\newblock H. Li, W. Yang, M. Wei, A. Wang, 
\newblock Dynamics in a diffusive predator–prey system with double Allee effect and modified Leslie–Gower scheme,
\newblock \emph{International Journal of Biomathematics}, \textbf{15(3)} (2022), 2250001.
\doilink{https://doi.org/10.1142/S1793524522500012}


\bibitem{ref39}
\newblock X. Hu, R. Sophia,
\newblock The role of host refuge and strong Allee effects in a host–parasitoid system,
\newblock \emph{International Journal of Biomathematics}, \textbf{16(5)} (2023), 2250107.
\doilink{https://doi.org/10.1142/S1793524522501078}

\bibitem{ref40}
\newblock J. Geng, Y. Wang, Y. Liu, et al.,
\newblock Analysis of an avian influenza model with Allee effect and stochasticity,
\newblock \emph{International Journal of Biomathematics}, \textbf{16(6)} (2023), 2250111.
\doilink{https://doi.org/10.1142/S179352452250111X}

\bibitem{ref41}
\newblock X. Xu, Y. Meng, Y. Shao,
\newblock Hopf bifurcation of a delayed predator–prey model with Allee effect and anti-predator behavior,
\newblock \emph{International Journal of Biomathematics}, \textbf{16(7)} (2023), 2250125.
\doilink{https://doi.org/10.1142/S179352452250125X}




\bibitem{ref26}
\newblock J.B. Ferdy, J. Molofsky,
\newblock Allee Effect, Spatial Structure and Species Coexistence,
\newblock \emph{Journal of Theoretical Biology}, \textbf{217(7)} (2010), 3542--3556.
\doilink{https://doi.org/10.1016/j.amc.2010.09.029}

\bibitem{ref27}
\newblock Z. Zhang, T. Ding, W. Huang, Z. Dong, 
\newblock \emph{Qualitative Theory of Differential Equation}, Science Press: Beijing, China, 1992; Volume 101. (In Chinese).


\bibitem{ref28}
\newblock L. Perko,
\newblock \emph{Differential Equations and Dynamical Systems}, Springer, New York, 1996. 
\doilink{https://doi.org/10.1007/978-1-4684-0392-3}

\bibitem{ref29}
\newblock A. Dhooge, W. Govaerts, Y.A. Kuznetsov,
\newblock MATCONT: A matlab package for numerical bifurcation analysis of odes,
\newblock \emph{ACM Transactions on Mathematical Software}, \textbf{29(2)} (2003), 141--164.
\doilink{https://doi.org/10.1145/779359.779362}

\bibitem{ref30}
\newblock D. Henry,
\newblock \emph{Geometric Theory of Semilinear Parabolic Equations}, Springer Berlin, Heidelberg, 2006.
\doilink{https://doi.org/10.1007/BFb0089647}


\bibitem{ref31}
\newblock V. Srivastava, E.M. Takyi, R.D. Parshad,
\newblock The effect of fear on two species competition, 
\newblock \emph{Mathematical Biosciences and Engineering}, \textbf{20(5)} (2023), 8814--8855.
\doilink{https://doi.org/10.3934/mbe.2023388}


\bibitem{ref32}
\newblock S. Chen, F. Chen, V. Srivastava, R.D. Parshad,
\newblock Dynamical analysis of a Lotka-Volterra competition model with both Allee and fear effect,
\newblock \emph{International Journal of Biomathematics (In press)}  (2023).

\bibitem{m20}
\newblock Maciel, G., Cosner, C., Cantrell, R. S., \& Lutscher, F., 
\newblock Evolutionarily stable movement strategies in reaction–diffusion models with edge behavior. 
\newblock \emph{Journal of mathematical biology}, 80, 61-92, (2020).

\bibitem{v23} 
\newblock Srivastava, V., Antwi-Fordjour, K and Parshad, R. D.,
 \newblock Exploring unique dynamics in a predator-prey model with generalist predator and group defense in prey. 
\newblock \emph{ CHAOS}, (under minor revision) (2023).

\bibitem{P16} 
\newblock Parshad, R. D., Quansah, E., Black, K., \& Beauregard, M.
\newblock Biological control via “ecological” damping: an approach that attenuates non-target effects. 
\newblock \emph{ Mathematical biosciences}, 273, 23-44, (2016). 


\end{thebibliography}
\end{document}